\pdfoutput=1
\documentclass[12pt]{article}

\usepackage[margin=1in]{geometry}

\usepackage{mathtools, mathrsfs, amssymb, stmaryrd}
\usepackage{graphicx, xcolor, tikz-cd, tabularx}
\usepackage{spectralsequences}

\usepackage{mathpazo}

\usepackage{enumitem}
\setlist[enumerate]{noitemsep}
\setlist[enumerate, 1]{label = (\alph*)}
\setlist[itemize]{noitemsep}

\usepackage{amsthm, hyperref}
\usepackage[capitalize, noabbrev]{cleveref}

\numberwithin{equation}{subsection}
\crefname{equation}{}{}

\newtheorem{theorem}[equation]{Theorem}
\newtheorem{lemma}[equation]{Lemma}
\newtheorem{proposition}[equation]{Proposition}
\newtheorem{corollary}[equation]{Corollary}
\newtheorem{introtheorem}{Theorem}
 
\crefname{introtheorem}{Theorem}{Theorems}

\theoremstyle{definition}
\newtheorem{definition}[equation]{Definition}
\newtheorem{remark}[equation]{Remark}

\newtheorem{notation}[equation]{Notation}

\newlist{propenum}{enumerate}{1}
\setlist[propenum]{label=(\alph*),ref=\theproposition\,(\alph*)}
\crefname{propenumi}{Proposition}{propositions}

\newcommand{\bbZ}{\mathbb{Z}}

\newcommand{\bbF}{\mathbb{F}}

\newcommand{\bbE}{\mathbb{E}}

\newcommand{\cP}{\mathcal{P}}

\DeclareMathOperator{\coker}{coker}

\DeclareMathOperator{\res}{res}

\DeclareMathOperator{\Ext}{Ext}

\DeclareMathOperator{\Gal}{Gal}

\newcommand{\Mod}{\textrm{Mod}}

\newcommand{\Perf}{\text{Perf}}

\newcommand{\Mack}{\textrm{Mack}}

\DeclareMathOperator{\hofib}{hofib}

\newcommand{\ul}{\underline}
\newcommand{\olboxempty}{{\overline\boxempty}}
\newcommand{\olboxslash}{{\overline\boxslash}}
\newcommand{\olcirc}{{\overline\circ}}
\newcommand{\HZ}{H\underline{\mathbb{Z}}}
\newcommand{\wt}{\widetilde}

\newcommand{\EP}{{E\cP}}
\newcommand{\tEP}{\wt{E}\cP}
\newcommand{\tEG}{\wt{E}G}

\newcommand{\wH}{\widehat{H}}

\renewcommand\epsilon\varepsilon

\renewcommand\emptyset\varnothing

\title{The Galois-equivariant $K$-theory of finite fields}
\author{David Chan \and Chase Vogeli}
\date{}

\begin{document}

\maketitle
\begin{abstract}
	We compute the $RO(G)$-graded equivariant algebraic $K$-groups of a finite field with an action by its Galois group $G$. Specifically, we show these $K$-groups split as the sum of an explicitly computable term and the well-studied $RO(G)$-graded coefficient groups of the equivariant Eilenberg--MacLane spectrum $\HZ$. Our comparison between the equivariant $K$-theory spectrum and $\HZ$ further shows they share the same Tate spectra and geometric fixed point spectra. In the case where $G$ has prime order, we provide an explicit presentation of the equivariant $K$-groups. 

    \vspace{\medskipamount}\noindent \textbf{MSC 2020}: 19D50 (Primary), 19L47, 55P91, 55Q91 (Secondary)
\end{abstract}

{\small\tableofcontents} 

\section{Introduction} \label{sec: introduction} 

    Algebraic $K$-theory is an important invariant of rings which provides a natural home for constructions in a range of subjects, from number theory to geometric topology. Classically, low dimensional $K$-groups of rings were defined first in terms of concrete algebraic constructions in work of Grothendieck, Bass, Milnor, and others.  Over time it was realized that these groups fit into a larger picture: there should be a sequence of such groups, indexed on the natural numbers.  This idea was first implemented in the work of Quillen, who constructed a space whose homotopy groups agreed with the already-known $K$-groups and whose higher homotopy groups give the correct generalization of ``higher algebraic $K$-theory.'' Using this machinery, Quillen gave a complete computation of the higher algebraic $K$-theory of finite fields \cite{Quillen:finite-field}.

    Despite Quillen's success in computing the algebraic $K$-theory of finite fields, the task of computing higher algebraic $K$-theory for other rings remains difficult. Indeed, the algebraic $K$-theory of the integers is still not entirely known, and its computation would resolve the longstanding Kummer--Vandiver conjecture in number theory \cite{Kurihara}.

    The work of this paper exists in the context of \emph{equivariant algebraic $K$-theory}, a variant of algebraic $K$-theory defined for rings with an action by a group $G$.  The most familiar examples of such rings come to us from Galois theory, where the group $G$ is the Galois group of a field extension.  Work of Merling \cite{Mona:Grings} provides a construction of algebraic $K$-theory for rings with $G$-action which produces a \emph{genuine $G$-spectrum} -- an enhancement of a spectrum with $G$-action -- with many desirable properties.  In particular, Merling shows that this construction is naturally related to the (now proven) Quillen--Lichtenbaum conjecture and provides a natural home for studying the $K$-theory of Galois extensions. 
    
    In addition to Merling's work, Barwick and Barwick--Glasman--Shah provide a different construction of equivariant algebraic $K$-theory using the language of $\infty$-categories and spectral Mackey functors \cite{Barwick,BGS}.  Malkiewich and Merling, in on-going work with Goodwillie and Igusa, are using techniques of equivariant algebraic $K$-theory to prove an equivariant refinement of Waldhausen's $A$-theory and the stable parametrized h-cobordism theorem \cite{Malkiewich--Merling1,Malkiewich--Merling2,Malkiewich--Merling--Goodwillie--Igusa}.  A connection between equivariant algebraic $K$-groups, equivariant $A$-theory, and Wall's finiteness obstruction is provided in work of the first-named author, Calle, and Mejia \cite{Calle--Chan--Mejia}.
    \subsection{Main result}
    
    One feature of genuine $G$-spectra is that their homotopy groups are graded on $RO(G)$, the Grothendieck group of real orthogonal representations of the group $G$.  The ordinary $\bbZ$-grading is recovered by considering the trivial $G$-representations of various dimensions. The computation of $RO(G)$-graded equivariant homotopy groups, even in tractable cases like Eilenberg--MacLane spectra, can be challenging. This work can be quite valuable, however.  For example, in Hill, Hopkins, and Ravenel's solution to the Kervaire invariant one problem they make critical use of these $RO(G)$-graded homotopy groups for the group $G=C_8$ \cite{HHR}.
    
    Let us fix a field $k$ with $G$-action and denote the equivariant algebraic $K$-theory of $k$ by $K_G(k)$.  Because $K_G(k)$ is a genuine $G$-spectrum, the algebraic $K$-groups are naturally graded on the group $RO(G)$. Our main result extends Quillen's computation of the $K$-groups of finite fields to a computation of the equivariant algebraic $K$-theory of finite fields with action by their Galois groups.

    \begin{introtheorem}[{\cref{thm: K groups split,prop: homotopy of the fiber}}]\label{thm: main thm intro}
        Let $k/\bbF_q$ be any finite extension of finite fields with Galois group $G$, and let $V$ be a virtual real orthogonal $G$-representation of dimension $|V|$.  Then the equivariant algebraic $K$-groups of $k$ split as 
        \[ 
            \pi^G_{V}(K_G(k))\cong
            \begin{cases}
                H^0(G;\pi_{|V|}K(k))\oplus \pi^G_{V}(\HZ), & |V|>0, \\
                \pi^G_{V}(\HZ), & |V|\leq 0,
            \end{cases} 
        \] 
        where $\HZ$ is the equivariant Eilenberg--MacLane spectrum on the constant Mackey functor $\underline{\bbZ}$. The $RO(G)$-graded ring structure is described in \cref{prop: products}.
    \end{introtheorem}
    
   This is one of the first systematic computations of equivariant algebraic $K$-groups.  Indeed, to our knowledge, the only prior computation of $RO(G)$-graded equivariant $K$-groups to appear in the literature is in the recent work of Elmanto and Zhang \cite{Elmanto--Zhang}.  There, the authors give a partial computation of the $RO(C_2)$-graded $K$-groups of finite quadratic field extensions, and show how these groups are related to the study of Artin $L$-functions.

    \subsection{Overview of key ideas}

    Our computation of the $RO(G)$-graded $K$-groups of finite fields hinges on a comparison of $K_G(k)$ with $\HZ$, the equivariant Eilenberg--MacLane spectrum on the constant $\bbZ$ Mackey functor.  The choice of $\ul{\bbZ}$ comes from the fact that there is an isomorphism of Mackey functors $\ul{\pi}_0(K_G(k))\cong \ul{\bbZ}$.  Accordingly, there is always a truncation map of genuine $G$-spectra $\tau_{\leq 0}\colon K_G(k)\to \HZ$ which realizes this isomorphism. The fiber of this map 
    \[
        \tau_{\geq 1} K_G(k) = \hofib(K_G(k) \xrightarrow{\tau_{\leq 0}} \HZ)
    \]
    is the \emph{1-connective cover} of $K_G(k)$.

    Our analysis of this map begins with the observation that it is fully computable after applying a construction called geometric completion. If $X$ is a $G$-spectrum, its \textit{geometric completion} is the mapping $G$-spectrum
    \[
        X^h = F(EG_+,X).
    \]
    The $G$-fixed points of $X^h$ are the familiar \textit{homotopy fixed points} $X^{hG}$ of $X$. As such, this is computable via a homotopy fixed-point spectral sequence which begins with group cohomology. We directly compute the following.

    \begin{introtheorem}[\cref{prop: homotopy of the fiber}] \label{thm: fiber intro version}
        For any extension of finite fields $k/\bbF_q$ with Galois group $G$, the 1-connective cover of the geometric completion of $K_G(k)$ has $RO(G)$-graded homotopy groups
        \[
            \pi^G_V \left( \tau_{\geq 1} K_G(k)^h \right) \cong 
            \begin{cases} 
                H^0(G;K_{|V|}(k)), & |V|>0, \\
                0 & \text{otherwise.}
            \end{cases}
        \]
    \end{introtheorem}

    The proof of \cref{thm: main thm intro} proceeds by showing geometric completion induces an equivalence on 1-connective covers
    \[ 
        \tau_{\geq 1} K_G(k) \xrightarrow{\sim} \tau_{\geq 1} K_G(k)^h,
    \]
    which yields the following result.
    
    \begin{introtheorem}[\cref{thm: pullback theorem}]\label{thm: pullback thm intro version}
        For any extension of finite fields $k/\bbF_q$ with Galois group $G$ there is a homotopy pullback of genuine $G$-spectra
        \[
            \begin{tikzcd}
                K_G(k) \ar[r, "\tau_{\leq 0}"] \ar[d] & \HZ \ar[d]\\
                K_G(k)^h \ar[r,  "\tau_{\leq 0}"] & \HZ^h.
            \end{tikzcd}
        \]
    \end{introtheorem}

    In conclusion, we have a fiber sequence of $G$-spectra
    \begin{equation} \label{eq: intro fibration sequence}
        \tau_{\geq 1} K_G(k) \to K_G(k) \xrightarrow{\tau_{\leq 0}} \HZ
    \end{equation}
    where we completely understand the homotopy groups of the fiber. This fiber sequence gives the decomposition in \cref{thm: main thm intro}.

    Our work reduces the $RO(G)$-graded $K$-groups of finite fields to the computation of the $RO(G)$-graded coefficient ring of $\HZ$. While these groups are complicated, they are known in many cases where $G$ is a finite cyclic group -- the only possibility for a Galois extension of finite fields. The first such computation appears in unpublished work of Stong for $G=C_p$, and is recorded in Ferland--Lewis \cite{Ferland--Lewis}. The computation for $G=C_{2^n}$ plays an important role in the work of Hill--Hopkins--Ravenel \cite{HHR,HHR:C4}. Work of Zeng handles the case of $G=C_{p^2}$ \cite{Zeng}, and work of Georgakopoulos revisits this case as well as providing a computer program for $G=C_{p^n}$ \cite{Georgakopoulos}. A recent preprint of Basu--Dey addresses the general case of $G=C_n$ \cite{Basu--Dey}. In \cref{sec: examples} we specialize to the case $G = C_{\ell}$, for $\ell$ a prime, and write out the $K$-groups explicitly.

    \begin{remark}
        The statement here of \cref{thm: main thm intro} references the homotopy \emph{groups} of $K_G(k)$. However, associated to any $G$-spectrum are its homotopy \emph{Mackey functors}, a richer invariant which contains the homotopy groups of all fixed-point spectra. Although it is more data to keep track of, we will see that the additional structure present in Mackey functors dramatically simplify calculations. In particular, it is essential to our approach to resolving extension problems that appear throughout the paper.  See \cref{remark: Mackey functors} for more details.  The Mackey functor structure is also critical in our analysis of the multiplicative structures in \cref{section: multiplicative structure}, specifically in the proof of \cref{prop: products}.
        
        For this reason, we work almost exclusively with the homotopy Mackey functors of $G$-spectra in the body of the paper.
    \end{remark}

\subsubsection{Proof of \cref{thm: pullback thm intro version}}

    Let us briefly outline our approach to proving \cref{thm: pullback thm intro version}. The main method we employ is a comparison of the \emph{Tate squares} associated to the $G$-spectra $K_G(k)$ and $\HZ$.  To any genuine $G$-spectrum $X$ there is a functorially associated homotopy pullback diagram of the form
    \[
        \begin{tikzcd}
            X\ar[r] \ar[d] & \wt{X} \ar[d]\\
            X^h \ar[r]  & X^t
        \end{tikzcd}
    \]
    where $X^t$ is called the \emph{Tate $G$-spectrum} of $X$.  As the name suggests, the homotopy groups of $X^t$ can computed using a spectral sequence whose $E^2$-page is computed using Tate cohomology. The map $K_G(k)\to \HZ$ induces a map on Tate diagrams and the following theorem is proved using direct computation.
    \begin{introtheorem} [\cref{prop: Tate equivalence} and \cref{prop: tilde equivalence}] \label{thm: Tate tilde intro version}
        For any extension of finite fields $k/\bbF_q$ with Galois group $G$ the map $K_G(k)\to \HZ$ induces equivalences of genuine $G$-spectra 
        \[\wt{K_G(k)}\to \wt{\HZ}\quad \mathrm{and}\quad K_G(k)^t\to \HZ^t.\]
    \end{introtheorem}

    Together with the naturality of the Tate square, this is enough to prove \cref{thm: pullback thm intro version}.

\subsubsection{Geometric fixed points}

    As a further demonstration of our methods, we show $K_G(k)$ and $\HZ$ have the same geometric fixed points. Using \cref{thm: main thm intro}, and the computation of the homotopy groups of the fiber, we prove the following.

    \begin{introtheorem}[\cref{cor: same geometric fixed points,prop: geometric fixed points of HZ}]\label{thm: geometric fixed points intro version}
    There is an equivalence of spectra 
    \[
        \Phi^G(K_G(k)) \xrightarrow{\sim} \Phi^G(\HZ).
    \]  
    The homotopy groups of these spectra as graded rings are given by
    \[
        \pi_*(\Phi^G(\HZ))\cong \begin{cases}
            \bbZ/p[x], & \textrm{$G\cong C_{p^n}, n>0,$}\\
            0, & \textrm{else,}
        \end{cases}
    \]
    where $|x|=2$.
    \end{introtheorem}
    We note that the computation of the groups $\pi_*(\Phi^G(\HZ))$ is well-known to experts, but we did not know a reference in the literature so the computation is given in \cref{prop: geometric fixed points of HZ} below.
    
    \subsubsection{Multiplicative structure}
    
    When $R$ is a commutative ring with $G$-action, work of Barwick--Glasman--Shah implies that the $RO(G)$-graded homotopy groups $\pi_{\star}^G(K_G(R))$ form an $RO(G)$-graded ring.  For finite field $k$ with Galois action by $G$, we show that the $RO(G)$-graded ring $\pi_{\star}^G(K_G(k))$ is a square-zero extension of $\HZ^G_{\star}$ in \cref{theorem: square zero extension}.  In the case where $G= C_2$, we use this to give a presentation of the ring $\pi_{\star}^{C_2}(K_{C_2}(k))$ in \cref{theorem: ring structure for C2}.

\subsection{Future directions}

    We end the introduction with a short discussion of possible future computations for equivariant algebraic $K$-theory of other Galois extensions. To any Galois extension $E/F$ of fields there is a natural map
    \[
        K(F)\to K(E)^{hG}.
    \]
    Underlying the analysis in this paper is the fact that for finite fields, this map is a connective cover.  We prove a version of this statement on the level of $G$-spectra in \cref{prop: K theory connected cover}.  This behavior, a version of Galois descent, is atypical of Galois extensions of rings; indeed we are aware of no other examples where this for field extensions.  In particular, a direct application of the techniques in this paper is unlikely to apply in other examples.  

    On the other hand, the work of Clausen--Mathew--Naumann--Noel \cite{CMNN} (see also \cite{Thomason,Thomason--Trobaugh}) shows that algebraic $K$-theory satisfies Galois descent after chromatic localization at the spectra $T(n)$.  Given this, we expect that after $T(n)$-localization it should be possible to use the ideas of this paper for further computations of equivariant algebraic $K$-theory for other Galois ring extensions.  We intend to explore this idea in future work.

\subsection*{Outline}

    The paper is organized as follows. In \cref{sec: Tate} we recall the Tate diagram associated to a $G$-spectrum, which serves as a guiding framework to our work. In \cref{sec: equivariant algebra} we recall the definition of a Mackey functor and prove results concerning extensions of Mackey functors. In \cref{sec: Galois action} we make a number of group (co)homology computations which provide the input to homotopy fixed point and Tate spectral sequences. These computations set the stage for the proofs of the main theorems which are in \cref{sec: proof}. The multiplicative structure in $\ul{\pi}_{\star}(K_G(k))$ is discussed in \cref{section: multiplicative structure}. We conclude the paper in \cref{sec: examples} with explicit computations in the case $G$ is a cyclic group of prime order.

\subsection*{Acknowledgments}

    The authors would like to thank Liam Keenan, Tomas Mejia Gomez, J.D. Quigley, and Ben Spitz for helpful conversations. We would additionally like to thank Maxine Calle, Teena Gerhardt, Maximilien P\'eroux, and Inna Zakharevich for reading parts of earlier drafts of this paper. We also thank an anonymous referee for several helpful suggestions, and in particular a simplification of the proof of \cref{thm: geometric fixed points intro version}. The first-named author was partially supported by NSF grant DMS-2135960.

\section{The Tate diagram} \label{sec: Tate}

In this short section we recall the Tate diagram from \cite{Greenlees--May}, which forms the basis for our work in later sections. For any finite group $G$ let $EG$ be a contractible $CW$-complex with free $G$-action. An explicit model is the unit sphere in the $G$-representation obtained by taking the infinite direct sum of copies of the reduced regular representation of $G$.  Define a space $\tEG$ by the cofiber sequence
\[
    EG_+\xrightarrow{\epsilon} S^0\to \tEG
\]
and note that for any $G$-spectrum $X$ we have an induced cofiber sequence
\[
    EG_+\wedge X\to X\to \tEG \wedge X
\]
which is natural in $X$.

We write $X^h : = F(EG_+,X)$, where $F$ denotes the internal mapping spectrum, and call this spectrum the \emph{geometric completion} of $X$.  The map $\epsilon$ induces a map $X\cong  F(S^0,X)\to X^h$ which gives a map on cofiber sequences
\[
\begin{tikzcd}
	EG_+\wedge X \ar[r] \ar[d] & X \ar[d] \ar[r] & \tEG\wedge X \ar[d]\\
	EG_+\wedge X^h \ar[r]  & X^h \ar[r] & \tEG\wedge X^h 
\end{tikzcd}
\]
called the \emph{Tate diagram}. The key property of the Tate diagram is the following.

\begin{theorem}[{\cite[Proposition 1.2]{Greenlees--May}}]
	The left vertical map in the Tate diagram is an equivalence.  Equivalently, the right square is a homotopy pullback of $G$-spectra.
\end{theorem}

For the remainder of this paper we use the abbreviations
\begin{align*}
	X_h & := EG_+\wedge X\simeq EG_+\wedge X^h, \\
	X^t & := \tEG\wedge X^h, \\
	\wt{X} & := \tEG\wedge X, 
\end{align*}
so that the Tate diagram takes the abbreviated form
\[\begin{tikzcd}
	{X_h} & X & {\wt{X}} \\
	{X_h} & {X^h} & {X^t}
	\arrow[from=1-1, to=1-2]
	\arrow["\simeq"', from=1-1, to=2-1]
	\arrow[from=1-2, to=1-3]
	\arrow[from=1-2, to=2-2]
	\arrow["\lrcorner"{anchor=center, pos=0.125}, draw=none, from=1-2, to=2-3]
	\arrow[from=1-3, to=2-3]
	\arrow[from=2-1, to=2-2]
	\arrow[from=2-2, to=2-3]
\end{tikzcd}.\]

The bottom row of the Tate diagram is called the \emph{norm cofiber sequence}. The three $G$-spectra in the norm cofiber sequence come equipped with spectral sequences which compute their $RO(G)$-graded homotopy groups. These spectral sequences are recalled in \cref{subsec: norm cofib computation} and are used to compute the homotopy groups appearing in the norm cofiber sequence for equivariant algebraic $K$-theory spectra. From there, the computation of $RO(G)$-graded $K$-groups is achieved by comparing the top and bottom rows of the Tate diagram.

We end this section with a lemma we will need in \cref{sec: proof}.
\begin{lemma}\label{lemma: pullback lemma}
	Suppose that $f\colon X\to Y$ is a map of genuine $G$-spectra which induces equivalences $\wt{X}\to \wt{Y}$ and $X^t\to Y^t$.  Then the square
	\[
	\begin{tikzcd}
		X \ar[r,"f"] \ar[d] & Y \ar[d]\\
		X^h \ar[r,"f^h"] & Y^h
	\end{tikzcd}    
	\]
	is a homotopy pullback square.
\end{lemma}
\begin{proof}
	Consider the commutative diagram
	\[\begin{tikzcd}[sep = small]
		{X_h} && X \\
		& {Y_h} && Y \\
		{X_h} && {X^h} \\
		& {Y_h} && {Y^h}
		\arrow[from=4-2, to=4-4]
		\arrow[from=2-2, to=2-4]
		\arrow[from=3-1, to=3-3]
		\arrow["\simeq"'{pos=0.3}, from=2-2, to=4-2]
		\arrow["\simeq"', from=1-1, to=3-1]
		\arrow["{f_h}"', from=1-1, to=2-2]
		\arrow[from=1-1, to=1-3]
		\arrow["f"', from=1-3, to=2-4]
		\arrow[from=2-4, to=4-4]
		\arrow["{f^h}"', from=3-3, to=4-4]
		\arrow[from=1-3, to=3-3]
		\arrow["{f_h}"', from=3-1, to=4-2]
	\end{tikzcd}\]
	obtained by naturality of the Tate square.  It suffices to show that the induced map $\hofib(f)\to \hofib(f^h)$ is a weak equivalence.  Taking fibers of the labeled maps in the cube above produces a commutative square
	\[
	\begin{tikzcd}
		\hofib(f_h)\ar[r,"\simeq"] \ar[d,"\simeq"] & \hofib(f) \ar[d]
		\\
		\hofib(f_h)\ar[r,"\simeq"] & \hofib(f^h)
	\end{tikzcd}
	\]
	where the horizontal maps are equivalences by the assumptions on the map $f$.  Thus the map $\hofib(f)\to \hofib(f^h)$ is a weak equivalence which completes the proof. 
\end{proof}

\section{Extensions of Mackey functors} \label{sec: equivariant algebra}

    In equivariant stable homotopy theory over a finite group $G$, invariants of $G$-spectra most naturally take values in the category of \emph{Mackey functors}. A Mackey functor is a collection of abelian groups indexed on the subgroups of $G$ with \emph{transfer} and \emph{restriction} maps between them mimicking induction and restriction maps for representation rings. 
    
    In this section, we recall the definition of a Mackey functor and prove an Ext vanishing result for Mackey functors that will be used repeatedly to solve extension problems later in \cref{prop: no extensions}.

\begin{definition}
    A \emph{Mackey functor} $\ul M$ for a finite abelian group $G$ is the data of
    \begin{itemize}

        \item an abelian group $\ul M(G/H)$ for each $H\leq G$,
        \item a \textit{transfer} homomorphism $T^K_H\colon \ul M(G/H)\to \ul M(G/K)$ for $H< K\leq G$, 
        \item a \textit{restriction} homomorphism $R^K_H\colon \ul M(G/K)\to \ul M(G/H)$ for $H< K\leq G$, 
        \item an action of the \textit{Weyl group} $W_KH=K/H$ on $\ul M(G/H)$ for $H<K \leq G$, 
        
    \end{itemize}
    which are subject to axioms enforcing compatibility of the restrictions, transfers, and Weyl actions \cite[\S 1.1.1]{Bouc:GreenFunctors}.

    A map $\ul M\to \ul N$ between Mackey functors is a collection of maps of abelian groups $\ul M(G/H)\to \ul N(G/H)$ for $H\leq G$ which commute with restrictions, transfers, and Weyl group actions.
\end{definition}


    For any Mackey functor $M$, the Weyl group action on $\ul M(G/e)$ turns $\ul M(G/e)$ into a $G$-module. The functor taking a Mackey functor $M$ to the $G$-module $\ul M(G/e)$ admits left and right adjoints
    $$ \begin{tikzcd}
    \Mack_G \arrow[r] & \Mod_{\bbZ G} \arrow[l, bend right, "\ul L"'] \arrow[l, bend left, "\ul R"]
    \end{tikzcd} $$
    which we denote by $\ul L$ and $\ul R$, respectively. For a $G$-module $P$, $\ul L(P)$ is its \textit{orbit Mackey functor} and $\ul R(P)$ is its \textit{fixed point Mackey functor}.  Explicitly, we have
    \[
        \ul L(P)(G/H) = P_H,\quad \mathrm{and} \quad \ul R(P)(G/H) = P^H.
    \]
    There is a natural transformation $N\colon\ul L\Rightarrow \ul R$ which for a $G$-module $P$ at each level $G/H$ is given by the norm map $P_H\to P^H$,
    \[
        x \mapsto \sum\limits_{h\in H} h\cdot x
    \]
    for any $x\in P_H$. This transformation will arise in our computations of Mackey functor-valued group cohomology in \cref{subsec: Cn Mackey functors}.

    In our computations, we will often encounter extensions of Mackey functors with trivial underlying $G$-module by Mackey functors in the image of $R$. We show that all of these extensions must be trivial.

\begin{proposition}\label{prop: no extensions}
   For any $G$-Mackey functor $\ul M$ with $\ul M(G/e)=0$ and any $G$-module $N$, 
   \[
        \Ext^1(\ul M,\ul R(N))=0.
   \]
\end{proposition}

\begin{proof}
    We will show that any short exact sequence of Mackey functors
    \[
        0\to \ul R(N)\xrightarrow{f} \ul P\to \ul M\to 0
    \]
    must be split.  Note that since $\ul M(C_n/e)=0$, we must have $f_{C_n/e}\colon \ul R(N)(C_n/e)\to \ul P(C_n/e)$ is an isomorphism.  Let $\varphi\colon \ul P(C_n/e)\to \ul R(N)(C_n/e)$ denote the inverse.
    
    Since $\ul R$ is the right adjoint to evaluation at $C_n/e$, the map $\varphi$ uniquely determines a map of Mackey functors $\Phi\colon \ul P\to \ul R(N)$ which is given by $\varphi$ at level $C_n/e$.  Thus the composite $\Phi\circ f$ is the identity on level $C_n/e$, and therefore must be the identity on the entire Mackey functor. Thus every such short exact sequence is split and $\Ext^1(\ul M,\ul R(N))=0$.
\end{proof}

A dual argument can be applied to see that $\Ext^1(L(N),M)=0$ for any $G$-module $N$. 


\section{The Galois action on $K$-theory} \label{sec: Galois action}

In her thesis \cite{Mona:Grings}, Merling defines the algebraic $K$-theory of rings with $G$-action.  This invariant assigns a genuine $G$-spectrum $K_G(R)$ to every ring $R$ with action of a finite group $G$.  The following theorem of Merling tells us how $K_G$ behaves when our $G$-rings arise from Galois extensions; we refer the reader to Merling's paper \cite{Mona:Grings} for proofs and further discussion.

\begin{theorem}[{\cite[Theorem 1.2, (4) and (6)]{Mona:Grings}}]\label{thm: fixed points of K}
	For a finite Galois extension of fields $E/F$ with $G = \mathrm{Gal}(E/F)$ there are equivalences of spectra
	\[
	K(F)\simeq K_G(E)^{G}
	\]
	where $K(F)$ is the non-equivariant algebraic $K$ theory of $F$.
\end{theorem}

We are interested in Merling's construction when the ring $R = k$ is a finite field which is a finite Galois extension of a finite field with $G$ the Galois group. Specifically, let $p$ be a prime, let $q = p^r$ for some $r\geq 1$, and let $\bbF_{q}$ be the field with $q$ elements. Quillen showed $\bbF_q$ has its higher $K$-groups concentrated in odd degrees, where we have
\[ 
K_{2i-1}(\bbF_q) \cong \bbZ/(q^i-1)
\]
for $i>0$ \cite[Theorem 8(i)]{Quillen:finite-field}. We will consider the case where $k=\bbF_{q^n}$ is a degree $n$ extension of $\bbF_q$. The Galois group $G=\Gal(k/\bbF_q)$ is cyclic of order $n$, generated by the Frobenius automorphism $\varphi$ specified by $\varphi(x)=x^q$ of $k$. Quillen furthermore calculated the action of $G$ on the higher $K$-groups.

\begin{theorem}[{\cite[Theorem 8(iii)]{Quillen:finite-field}}] \label{theorem: Quillen}
	Let $k = \bbF_{q^n}$ and let $G = \mathrm{Gal}(k/\bbF_q)$.  For $i>0$, the action of the Frobenius automorphism $\varphi$ on the group
	$$ K_{2i-1}(k) = \bbZ/(q^{ni}-1) $$
	is multiplication by $q^i$.
\end{theorem}

In this section, we use these two results to preform the calculations underlying our main results. Our goal is to explicitly compute the $RO(G)$-graded homotopy Mackey functors of the terms appearing in the norm cofiber sequence for $K_G(k)$. In \cref{subsec: Cn Mackey functors}, we enumerate all the possible Mackey functors appearing on the $\ul E^2$-pages of spectral sequences converging to these terms, and in \cref{subsec: norm cofib computation}, we analyze these spectral sequences.

\subsection{Galois cohomology Mackey functors} \label{subsec: Cn Mackey functors}

In this section, we determine the Mackey functors appearing on the $\ul E^2$-pages of the homotopy orbit and homotopy fixed point spectral sequences for $K_G(k)$.  The $\ul E^2$-pages of these spectral sequences are recalled in \cref{subsec: norm cofib computation} below.  We first extend \cref{theorem: Quillen} to a compuation of the $G$-modules $\pi^e_V K_G(k)$ for $V\in RO(G)$, and then we obtain Mackey functors by applying the functors
\[
    \ul L,\ul R\colon \Mod_{\bbZ G}\to \Mack_G
\]
defined in the last subsection.

Since these functors are given by taking orbits and fixed points respectively, they form the input for the filtration-zero line of the homotopy orbit and fixed point spectral sequences.  The remainder of the input for these spectral sequences amounts to computing the higher group (co)homology of these $G$-modules. Since $G$ is cyclic, its (co)homology with coefficients in any $G$-module $P$ is 2-periodic, and is fully determined by the norm map $N\colon P_G\to P^G$. Specifically, we have
\[
    H_s(G;P) \cong \begin{cases}
        P_G & s=0, \\
        \coker N_P & \text{$s>0$, $s$ odd,} \\
        \ker N_P & \text{$s>0$, $s$ even,}
    \end{cases} \qquad
    H^s(G;P) \cong \begin{cases}
        P^G & s=0, \\
        \ker N_P & \text{$s>0$, $s$ odd,} \\
        \coker N_P & \text{$s>0$, $s$ even.}
    \end{cases}
\]
Thus, we need to compute the kernel and cokernel of the norm transformation $N\colon \ul L\Rightarrow \ul R$.

Recall that $K_{G}(k)^e\simeq K(k)$ \cite[Theorem 6.4]{Mona:Grings}, so we have an isomorphism of abelian groups $\pi^e_V K_G(k)\cong K_{|V|}(k)$. The $G$-action on the group $\pi^e_V K_G(k)$ is determined by Quillen's computation \cref{theorem: Quillen} and the $G$-action on the representation sphere $S^V$. We introduce terminology to classify the possible $G$-actions on representation spheres.

\begin{notation}
    We call a $G$-representation $V$ \emph{orientation-preserving} (resp. \emph{reversing}) if the induced map on $S^V$ is orientation-preserving (resp. reversing.). By extending linearly to virtual representations, this gives a well defined group homomorphism $RO(G)\to \bbZ/2$ where the kernel is all orientation preserving virtual representations. Note that if $G$ is a cyclic group of odd order, all representations are orientation-preserving.
\end{notation}

An orientation-reversing action of $G$ on $S^V$ introduces a ``twist'' of $-1$ to the $G$-action on $\pi^e_V K_G(k)$. We use the notation $\bbZ^\sigma$ to denote the sign representation, i.e., the group $\bbZ$ with $G$ acting by $-1$. For a $G$-module $M$, we write 
$$ M^\sigma = M \otimes \bbZ^\sigma $$
for $M$ twisted by the sign action. We write $K_{2i-1}(k)$ for the $G$-module $\bbZ/(q^{ni}-1)$ with the implicit $G$-action by $q^i$ as in \cref{theorem: Quillen}.

\begin{lemma} \label{lemma:UnderlyingGmodules}
Let $k$ be any finite field.  For $V\in RO(G)$, $G$-modules $\pi^e_V K_G(k)$ are given by
$$ \pi^e_V K_G(k) \cong \begin{cases} 
\bbZ & \textrm{$|V|=0$, $V$ orientation-preserving,} \\
\bbZ^\sigma & \textrm{$|V|=0$, $V$ orientation-reversing,} \\
K_{2i-1}(k) & \textrm{$|V|=2i-1>0$, $V$ orientation-preserving,} \\
K_{2i-1}(k)^\sigma & \textrm{$|V|=2i-1>0$, $V$ orientation-reversing,} \\
0 & \textrm{otherwise.}
\end{cases} $$
\end{lemma}
\begin{proof}
    The action on the homotopy groups is given by the conjugation action on maps, where $G$ acts on $S^V$ by sign when $V$ is orientation-reversing and trivially when $V$ is orientation-preserving. 
    
    In virtual dimension $|V|=2i-1$ for $i>0$, the claimed computation follows from \cref{theorem: Quillen}. In virtual dimension $|V|=0$, the action of $G$ on $K_0(k)\cong\bbZ$ is necessarily trivial, as there are no nontrivial ring automorphisms of $\bbZ$. Therefore, the $G$-module $\pi_V^e K_G(k)$ is either $\bbZ$ or $\bbZ^\sigma$, depending on whether $V$ is orientation-preserving or orientation-reversing.
\end{proof}

Writing $M$ for any of the $G$-modules in \cref{lemma:UnderlyingGmodules}, the rest of this section is devoted to computing the Mackey functors in the exact sequence:
\[
    0\to \ker(N)\to \ul L(M)\xrightarrow{N} \ul R(M)\to \coker(N)\to 0.
\]
We treat the four cases of \cref{lemma:UnderlyingGmodules} in order. A summary of these calculations can be found in \cref{table: symbols}.

\subsubsection{The trivial $G$-module $\bbZ$.} 

We denote by 
\begin{align*} 
    \ul R(\bbZ) & = \ul\bbZ   = \boxempty, \\
    \ul L(\bbZ) & = \boxslash,
\end{align*} 
the fixed points and orbit Mackey functors, respectively, of trivial $C_n$-module $\bbZ$. We define the Mackey functor $\circ$ as the cokernel of the norm map. This is depicted by the short exact sequence of Mackey functors

$$ \begin{tikzcd}[row sep = small]
& 0 \arrow[r] & \boxslash \arrow[r] & \boxempty \arrow[r] & \circ \arrow[r] & 0 \\
C_n/C_a: & 0 \arrow[r] \arrow[dd, bend right = 20] & \bbZ \arrow[r, "a"] \arrow[dd, bend right = 20, "a/b"'] & \bbZ \arrow[r, two heads] \arrow[dd, bend right = 20, "1"'] & \bbZ/a \arrow[r] \arrow[dd, bend right = 20, "1"'] & 0 \arrow[dd, bend right = 20] \\
& & & & & \\
C_n/C_b: & 0 \arrow[r] \arrow[uu, bend right = 20] & \bbZ \arrow[r, "b"] \arrow[uu, bend right = 20, "1"'] & \bbZ \arrow[r, two heads] \arrow[uu, bend right = 20, "a/b"'] & \bbZ/b \arrow[r] \arrow[uu, bend right = 20, "a/b"'] & 0. \arrow[uu, bend right = 20] \\
\end{tikzcd} $$

\subsubsection{The sign representation $\bbZ^\sigma$.} 

In the case where $n$ is even, there is the $C_n$-module $\bbZ^\sigma$. We denote by 
\begin{align*} 
    \ul R(\bbZ^\sigma) & = \olboxempty, \\
    \ul L(\bbZ^\sigma) & = \olboxslash,
\end{align*} 
the fixed points and orbit Mackey functors, respectively. A subgroup $C_m\leq C_n$ acts nontrivially on $\bbZ^\sigma$ if and only if its index $n/m$ is odd. The norm map $L(\bbZ^\sigma)\to R(\bbZ^\sigma)$ has both a nontrivial kernel $\bullet$ and cokernel $\olcirc$, which have components given by:
\begin{align*} 
\olboxslash(C_n/C_m) & \cong \begin{cases}
\bbZ/2 & \textrm{$n/m$ odd} \\
\bbZ^\sigma & \textrm{$n/m$ even,} 
\end{cases}
&
\olboxempty(C_n/C_m) & \cong \begin{cases}
0 & \textrm{$n/m$ odd} \\
\bbZ^\sigma & \textrm{$n/m$ even,}
\end{cases} \\
\bullet(C_n/C_m) & \cong \begin{cases}
\bbZ/2 & \textrm{$n/m$ odd} \\
0 & \textrm{$n/m$ even,} 
\end{cases}
&
\olcirc(C_n/C_m) & \cong \begin{cases}
0 & \textrm{$n/m$ odd} \\
(\bbZ/m)^\sigma & \textrm{$n/m$ even.}
\end{cases}
\end{align*}

\subsubsection{Higher $K$-groups with the Galois action.} 

Recall that the $G$-module $K_{2i-1}(k)$ is given by $\bbZ/(q^{ni}-1)$ with $G$ acting by $q^i$.

\begin{lemma} \label{lemma:ominus}
    Let $k= \bbF_{q^n}$ and $G = C_n$. The norm map 
    \[
        N\colon \ul L(K_{2i-1}(k)) \to \ul R(K_{2i-1}(k))
    \]
    is an isomorphism of $C_n$-Mackey functors, and thus has trivial kernel and cokernel.
\end{lemma}

    We denote the resulting Mackey functor by $\varominus^i$, defined as
    \[ 
        \varominus^i= \ul L(K_{2i-1}(k))\cong \ul R(K_{2i-1}(k)).
    \]
    It is specified by $\varominus^i(C_n/C_m) = \bbZ/(q^{ni/m}-1)$ with transfers and restrictions depicted in the Lewis diagram
    $$ \begin{tikzcd}[row sep = small]
    C_n/C_a\colon & \bbZ/(q^{ni/a}-1) \arrow[dd, bend right = 20, "\sum\limits_{j=0}^{a-1} q^{jni/a}"']  \\ \\
    C_n/C_b\colon & \bbZ/(q^{ni/b}-1) \arrow[uu, bend right = 20, "1"']  \\
    \end{tikzcd} $$
    The Weyl group action of $C_n/C_m$ on $\bbZ/(q^{ni/m}-1)$ is given by multiplication by $q^{ni/m}$.

\begin{proof}[Proof of \cref{lemma:ominus}]
    The coinvariants and invariants of $K_{2i-1}$ relative to the subgroup \[ C_m = \langle q^{n/m}\rangle \leq C_n, \] are computed as follows.

    The coinvariants $K_{2i-1}(k)_{C_m}$ are obtained by adding the relations $q^{ni/m}x=x$ for all $x$ to the group $\bbZ/(q^{ni}-1)$. In other words, this is the quotient by the subgroup generated by $q^{ni/m}-1\in\bbZ/(q^{ni}-1)$. This quotient is cyclic of order $q^{ni/m}-1$, generated by the coset of $1\in\bbZ/(q^{ni}-1)$.

    The invariants $K_{2i-1}(k)^{C_m}$ are the subgroup of $\bbZ/(q^{ni}-1)$ of elements $x$ such that $q^{ni/m}x=x$. In other words, it is the subgroup annihilated by $q^{ni/m}-1$, which is cyclic of order $q^{ni/m}-1$ and is generated by the element 
    \[ 
        \frac{q^{ni}-1}{q^{ni/m}-1} \in \bbZ/(q^{ni}-1). 
    \]

    Thus, we see that the invariants and coinvariants are abstractly isomorphic.  Moreover, the norm map is multiplication by
    \[ 
        N = 1 + q^{ni/m} + q^{2ni/m} + \cdots + q^{(m-1)ni/m} = \frac{q^{ni}-1}{q^{ni/m}-1}.
    \]
    Under the above identifications, this takes the generator of $K_{2i-1}(k)_{C_m}$ to a generator of $K_{2i_1}(k)^{C_m}$, from which the claim follows.
\end{proof}

\subsubsection{Higher $K$-groups with the twisted Galois action.} 

Lastly, we consider the norm map of the twisted $G$-modules $K_{2i-1}(k)^{\sigma} = K_{2i-1}(k)\otimes \mathbb{Z}^{\sigma}$ in the case where $n$ is even. This case is almost identical to the last, but with $(-1)^{n/m}q^{ni/m}$ in place of $q^{ni/m}$.  We leave the details to the reader.

\begin{lemma} \label{lemma:NormOplus}
Recall that $C_n$ acts on $K_{2i-1}(k)^\sigma \cong \bbZ/(q^{ni}-1)$ by $-q^{i}$.  For $C_m\leq C_n$ we have
\begin{enumerate}
\item the coinvariants $(K_{2i-1}(k)^\sigma)_{C_m}$ are cyclic of order $(-1)^{n/m}q^{ni/m}-1$,
\item the invariants $(K_{2i-1}(k)^\sigma)^{C_m}$ are cyclic of order $(-1)^{n/m}q^{ni/m}-1$,
\item the norm map 
$$ N\colon (K_{2i-1}(k)^\sigma)_{C_m} \to (K_{2i-1}(k)^\sigma)^{C_m} $$ 
is an isomorphism.  
\end{enumerate}
\end{lemma}
Thus we have an isomorphism of Mackey functors $\ul L(K_{2i-1}(k)^{\sigma})\cong \ul R(K_{2i-1}(K)^{\sigma})$ and we write $\varoplus^i$ for this Mackey functor.

\subsubsection{Summary} \label{subsubsec: summary of Mackey functors}

The results of the computations above are summarized in \cref{table: symbols}.  For space reasons, we abbreviate ``orientation preserving'' (resp. reversing) to o.p. (resp. o.r.). The symbols for $C_n$-Mackey functors were chosen based on the following conventions:
\begin{itemize}
    \item A square symbol indicates the underlying abelian group is $\bbZ$.
    \item A circular symbol indicates each level is a finite cyclic group.
    \item A horizontal bar indicates the Weyl groups at even index levels act by sign.
\end{itemize}

\begin{table}[!htbp]
	\renewcommand{\arraystretch}{1.2}
    \centering
    \begin{tabular}{| c | c | c | c | c | c |}
        \hline 
            & ${K}_V(k)$ & $\ul L({K}_V(k))$ & $\ul R({K}_V(k))$ & $\ker(N)$ & $\coker(N)$ \\
        \hline 
        $|V|=0$, o.p. & $\bbZ$ & $\boxslash$ & $\boxempty$ & $0$ & $\circ$ \\
        \hline
        $|V|=0$, o.r. & $\bbZ^{\sigma}$ & $\olboxslash$ & $\olboxempty$ & $\bullet$ & $\olcirc$ \\
        \hline
        $|V|=2i-1>0$, o.p. & $\bbZ/(q^{ni}-1)$ & $\varominus^i$ & $\varominus^i$ & $0$ & $0$ \\
        \hline
        $|V|=2i-1>0$, o.r. & $\bbZ^{\sigma}/(q^{ni}-1)$ & $\varoplus^i$ & $\varoplus^i$ & $0$ & $0$ \\
        \hline
    \end{tabular}
    \caption{The $G$-modules $K_V(k)$ and associated Mackey functors} \label{table: symbols}
\end{table}

\subsection{The norm cofiber sequence}\label{subsec: norm cofib computation}

In this section, we determine the $RO(G)$-graded homotopy Mackey functors of the $G$-spectra in the norm cofiber sequence
$$ \begin{tikzcd} K_G(k)_h \arrow[r,"N"] & K_G(k)^h \arrow[r] & K_G(k)^t, \end{tikzcd} $$
which forms the bottom row of the Tate diagram.  The main result is \cref{prop: Tate equivalence}, which says that the map $K_G(k)^t\to \HZ^t$ is a weak equivalence of $G$-spectra.

Our analysis is centered around the $RO(G)$-graded, Mackey functor-valued homotopy orbit spectral sequence (HOSS)
\[ 
    \ul E^2_{s,V} = \ul H_s(G;\pi^e_V K_G(k)) \Rightarrow \underline\pi_{s+V} K_G(k)_h, 
\]
and homotopy fixed point spectral sequence (HFPSS) \cite[\S 2.2]{BBHS} 
\[ 
    \ul E^2_{s,V} = \ul H^{-s}(G;\pi^e_V K_G(k)) \Rightarrow \underline\pi_{s+V} K_G(k)^h.
\]
The group homology Mackey functors $\ul H_*(G;M)$ appearing on the $\ul E^2$-page are given by
\[
    \ul H_*(G;M)(G/K) = H_*(K; \res^G_K M)
\]
for a $G$-module $M$. The group cohomology Mackey $\ul H^*(G;M)$ functors are defined analogously. The differentials have grading
\[ 
    d^r \colon \ul E^r_{s,V}\to \ul E^r_{s-r,V+r-1}.
\]

\begin{proposition}
The $RO(G)$-graded homotopy Mackey functors of $K_G(k)_h$ are given on orientation preserving representations $V$ by
\[
    \ul\pi_V K_G(k)_h \cong
    \begin{cases}
        \boxslash & |V|=0, \\
        \varominus^i \oplus \circ & |V|=2i-1>0, \\
        0 & \textrm{otherwise},
    \end{cases}
\]
and on orientation reversing $V$ by
\[
    \ul\pi_V K_G(k)_h \cong
    \begin{cases}
        \olboxslash & |V|=0 \\
        \varoplus^i \oplus \olcirc & |V|=2i-1>0, \\
        \bullet & |V|=2i>0, \\
        0 & \textrm{otherwise}.
    \end{cases}
\]
\end{proposition}
\begin{proof}
    Fix a $V\in RO(G)$ and let $W = V-|V|$.  The homotopy orbit spectral sequence, based at  $W$, has the form
    \[
        \ul{E}_{s,W+t}^2 = \ul{H}_{s}(G;\pi^e_{W+t}(K_G(k)))\Rightarrow \ul{\pi}_{W+s+t}(K_G(k))
    \]
    and we compute $\ul{\pi}_V(K_G(k))$ by looking at the line in this spectral sequence with $s+t=|V|$.  Note that by \cref{lemma:UnderlyingGmodules}, the Mackey functors in this spectral sequence depend only on the number $|W|+t$ and whether $W$ is orientation preserving or reversing.  Since $W$ and $V$ differ by a trivial representation they are either both orientation preserving or orientation reversing.

    The spectral sequences vanish for negative $s$ because negative group homology is zero.  When $t$ is negative, the spectral sequence vanishes because $K(k)$ is a connective spectrum.  Thus we have a first quadrant spectral sequence.  In fact this spectral sequence vanishes except when $s$ or $t$ is equal to zero.  
    
    To see this, suppose first that $W$ is orientation preserving. Note that when $t$ is not zero we have
    \[
        \ul E^2_{s,t+W} \cong \ul H_s(G;K_{t}(k))
    \]
    and this Mackey functor is zero since the groups $K_t(k)$ are either zero or the groups $\bbZ/(q^{in}-1)$ with $C_n$-action by $q^i$. These homology groups vanish for $s>0$ since,  by \cref{lemma:ominus}, the norm map from orbits to fixed points is an isomorphism. The case of $W$ orientation reversing is essentially the same.

    The spectral sequence is displayed in \cref{figure: HOSS}, with two cases depending on whether $W$ is orientation preserving or reversing.  There are no possible differentials when $W$ is orientation preserving.  There are possible differentials when $W$ is orientation reversing.  These differentials have the form $\bullet\to \varoplus^i$ for various $i$.  Note that these maps must be zero since $\bullet(C_n/e)=0$ and thus there are no non-zero maps of Mackey functors $\bullet\to \varoplus^i$ since $\varoplus^i$ is in the image of the functor $R$.  All possible extensions are trivial by \cref{prop: no extensions}.   
\end{proof}

\begin{figure}[ht]
\centering

\begin{sseqpage}[grid = crossword, 
    classes = {draw = none}, 
    title = {$\ul E^2_{s,W+t}$, $W$ o.p.},
    xrange = {0}{5},
    yrange = {0}{5},
    x label = $s$, 
    y label = $t+|W|$,]

\class["\boxslash"](0,0)

\foreach \i in {1,2,3} {
    \class["\varominus^\i"](0,2*\i-1)
}

\foreach \i in {1,3,5} {
    \class["\circ"](\i,0)
    \structline[dashed](\i,0)(0,\i)
}

\end{sseqpage} 
\begin{sseqpage}[grid = crossword, 
    classes = {draw = none}, homological Serre grading,
    title = {$\ul E^2_{s,W+t}$, $W$ o.r.},
    xrange = {0}{5},
    yrange = {0}{5},
    x label = $s$, 
    y label = $t+|W|$,]

\class["\olboxslash"](0,0)

\foreach \i in {1,2,3} {
    \class["\varoplus^\i"](0,2*\i-1)
}

\foreach \i in {1,3,5} {
    \class["\olcirc"](\i,0)
    \structline[dashed](\i,0)(0,\i)
}

\foreach \i in {2,4} {
    \class["\bullet"](\i,0)
    \d\i(\i,0)
}

\end{sseqpage}

\caption{The $\ul E^2$-pages of the homotopy orbit spectral sequence (HOSS) for $K_G(k)$. Dashed lines indicate potential extensions, and arrows indicate potential differentials. 
\label{figure: HOSS}}

\end{figure}

\begin{remark} \label{remark: Mackey functors}
    The Ext computations afforded by \cref{prop: no extensions} inform our choice to work with Mackey functors throughout this paper.  Even if one were interested only in the $C_n/C_n$-level of the computations, it is important to make use of the entire Mackey functor structure because it helps us to easily resolve extension problems which come from spectral sequences.  
    
    For a concrete example, consider the Galois extension $\bbF_9/\bbF_3$, with Galois group $C_2$, and $i=1$. If we tried to solve extension problems only at the level $C_2/C_2$, we would arrive at an extensions problem which, a priori, has two possible solutions. Indeed,
    \[ 
        \Ext^1_{\bbZ}(\circ(C_2/C_2),\varominus^1(C_2/C_2)) = \Ext^1_{\bbZ}(\bbZ/2,\bbZ/2) = \bbZ/2.
    \]  
    Of course, one can often resolve such extension problems by other means, but using the additional structure afforded by Mackey functors provides a systematic way to do so.
\end{remark} 

The next proposition follows from essentially the same arguments as the last, using the homotopy fixed points spectral sequence in place of the homotopy orbits spectral sequence.  This spectral sequence is displayed in \cref{figure: HFPSS}.  Note that in this case there is no room for either non-trivial differentials or extensions. 

\begin{proposition}\label{prop: homotopy fixed point groups}
The $RO(G)$-graded homotopy Mackey functors of $K_G(k)^h$ are given on orientation preserving representations $V$ by
\[
    \ul\pi_V K_G(k)^h \cong
    \begin{cases}
        \boxempty & |V|=0, \\
        \varominus^i & |V|=2i-1>0, \\
        \circ & |V|=2i<0, \\
        0 & \textrm{otherwise},
    \end{cases}
\]
and on orientation reversing $V$ by
\[
    \ul\pi_V K_G(k)^h \cong
    \begin{cases}
        \boxempty & |V|=0, \\
        \varoplus^i & |V|=2i-1>0, \\
        \bullet & |V|=2i-1<0 \\
        \olcirc & |V|=2i<0, \\
        0 & \textrm{otherwise}.
    \end{cases}
\]
\end{proposition}

\begin{figure}[ht]
\centering

\begin{sseqpage}[grid = crossword, 
    classes = {draw = none}, 
    title = {$\ul E^2_{s,W+t}$, $W$ o.p.},
    xrange = {-5}{0},
    xrange = {-5}{0},
    x label = $s$, 
    y label = $t+|W|$]

\class["\boxempty"](0,0)

\foreach \i in {1,2,3} {
    \class["\varominus^\i"](0,2*\i-1)
}

\foreach \i in {2,4} {
    \class["\circ"](-\i,0)
}

\end{sseqpage}
\begin{sseqpage}[grid = crossword, 
    classes = {draw = none},
    title = {$\ul E^2_{s,W+t}$, $W$ o.r.},
    xrange = {-5}{0},
    xrange = {-5}{0},
    x label = $s$, 
    y label = $t+|W|$]

\class["\olboxempty"](0,0)
\foreach \i in {1,2,3} {
    \class["\varoplus^\i"](0,2*\i-1)
}

\foreach \i in {1,3,5} {
    \class["\bullet"](-\i,0)
}

\foreach \i in {2,4} {
    \class["\olcirc"](-\i,0)
}

\end{sseqpage}

\caption{The $\ul E^2$-pages of the homotopy fixed point spectral sequence (HFPSS) for $K_G(k)$.}
\label{figure: HFPSS}
\end{figure}

With this computation in hand, we prove the first part of \cref{thm: Tate tilde intro version}

\begin{proposition}\label{prop: Tate equivalence}
    There is an equivalence of $G$-spectra $K_G(k)^t\xrightarrow{\simeq} \HZ^t$.
\end{proposition}
\begin{proof}
    We consider the map of Tate spectral sequences
    \[ \begin{tikzcd}
        \ul E^2_{s,V} = \ul\wH^s(G; \pi^e_V K_G(k)) \arrow[d] \arrow[r, Rightarrow] & \ul\pi_{s+V} K_G(k)^t \arrow[d] \\
        \ul E^2_{s,V} = \ul\wH^s(G; \pi^e_V \HZ) \arrow[r, Rightarrow] & \ul\pi_{s+V} \HZ^t
    \end{tikzcd} \]
    induced by the zeroth Postnikov section map $K_G(k)\to\HZ$, and show that there is an equivalence of $\ul E^2$-pages. 
    
    The effect of the computations in \cref{lemma:NormOplus,lemma:ominus} is that 
    \[
        \ul\wH^*(G; \pi^e_V K_G(k)) = 0
    \]
    for $|V| \neq 0$. Indeed, a $G$-module for which the norm map is an isomorphism has vanishing Tate cohomology.
    
    As a result, the $\ul E^2$-page for $K_G(k)$ is supported in the region where $|V|=0$. The $\ul E^2$-page for $\HZ$ is also supported in this range, since $\pi^e_\star \HZ$ is concentrated in total degree 0. Within the $|V|=0$ region of the $\ul E^2$-page, the zero Postnikov section map induces an isomorphism
    \[
        \ul\wH^*(G; \pi^e_V K_G(k)) \cong \ul\wH^*(G; \pi^e_V \HZ),
    \]
    which finishes the proof of the claim.
\end{proof}

\section{Proofs of the main theorems} \label{sec: proof}

As in the last section, we fix a field $k = \bbF_{q^n}$ for $q$ a positive power of a prime $p$ and let $G \cong C_n$ be the Galois group of the extension $k/\bbF_q$. In this section we prove the main results of the paper, starting with \cref{thm: pullback thm intro version} which we now recall.

\begin{theorem}\label{thm: pullback theorem}
    There is a homotopy pullback of genuine $G$-spectra
    \[
        \begin{tikzcd}
            K_G(k) \ar[r] \ar[d] & \HZ \ar[d]\\
            K_G(k)^h \ar[r] & \HZ^h
        \end{tikzcd}
    \]
    where the horizontal maps are the zeroth Postnikov truncations and the vertical maps are the geometric completion maps from the Tate diagram.
\end{theorem}

The homotopy groups of the bottom horizontal arrow can be computed effectively using the homotopy orbit spectral sequences. We use this to compute the homotopy groups of the fiber of both horizontal maps in \cref{prop: homotopy of the fiber} below.    

To prove \cref{thm: pullback theorem}, note that \cref{lemma: pullback lemma} reduces the theorem to proving that the map $K_G(k)\to \HZ$ induces equivalences 
\[
    \wt{K_G(k)}\to \wt{\HZ}\quad \mathrm{and}\quad K_G(k)^t\to \HZ^t.
\]
The equivalence on Tate spectra was proven in the previous section as \cref{prop: Tate equivalence}.  The fact that the map $\wt{K_G(k)}\to \wt{\HZ}$ is an equivalence then follows from the observation, \cref{cor: Tate connected cover} below, that the maps $\wt{K_G(k)}\to K_G(k)^t$ and $\wt{\HZ}\to \HZ^t$ are both connective covers.  

\subsection{Comparison of Tate diagrams}

\begin{proposition}\label{prop: K theory connected cover}
    The map $K_G(k)\to K_G(k)^h$ is a connective covering of $G$-spectra.
\end{proposition}
\begin{proof}
    By Merling's result, \cref{thm: fixed points of K}, we know that for all $C_m\leq C_n$ we have \[K_G(\bbF_{q^n})^{C_m}\cong K((\bbF_{q^n})^{C_m}) \cong K(\bbF_{q^{n/m}})\] 
    and so by Quillen's computation, \cref{theorem: Quillen}, we have
    \[
        \pi^{C_m}_i(K_G(k))\cong \pi_i(K(\bbF_{q^{n/m}}))\cong \bbZ/(q^{ni/m}-1)\cong \varominus^i(C_n/C_m)
    \]
    and so there are isomorphisms $\ul\pi_i(K_G(k))\cong \varominus^i$ for all $i\geq 0$.  Comparing with \cref{prop: homotopy fixed point groups} for $V=i$, we see that the positive integral homotopy Mackey functors are abstractly isomorphic and it remains to check that the map $K_G(k)\to K_G(k)^h$ from the Tate diagram actually induces an isomorphism.
    
    Since the map of underlying spectra $K_G(k)\to K_G(k)^h$ is an equivalence it gives an isomorphism at $C_n/e$ level of $\varominus^i$.  Since 
    \[
        \varominus^i\cong \ul R(\bbZ/q^{ni}-1),
    \]
    an endomorphism which is an isomorphism at the bottom level is an isomorphism.  Thus for all $i\geq 0$ the map $\ul\pi_i(K_G(k))\to \ul\pi_i(K_G(k)^h)$ is an isomorphism of Mackey functors.
\end{proof}

\begin{proposition}\label{prop: HZ connective cover}
    The map $\HZ\to \HZ^h$ is a connective covering of $G$-spectra.
\end{proposition}

\begin{proof}
    On integer-graded homotopy groups, we have 
    \[ 
        \ul\pi_s \HZ^h \cong \ul H^{-s}(G;\bbZ),
    \]
    which is concentrated in nonpositive degrees. In degree zero,
    \[ 
        \ul\pi_0 \HZ^h \cong \ul H^0(G;\bbZ) \cong \ul\bbZ,
    \]
    and $\ul\pi_0\HZ\to \ul\pi_0\HZ^h$ is an isomorphism.
\end{proof}

\begin{lemma}[{c.f.\ \cite[Lemma 11.2]{Greenlees--Meier}}]\label{lemma: Greenlees-Meier Connected cover lemma}
    For any connective $G$-spectrum $X$, the map $X\to X^h$ is a connective cover if and only if $\wt{X}\to X^t$ is a connective cover.
\end{lemma}
\begin{proof}
    Because $X$ is connective the homotopy orbit spectral sequence shows that $X_h$ is also connective.  Thus the Tate diagram induces a map of long exact sequences of homotopy Mackey functors which, near degree zero, looks like:
\[\begin{tikzcd}
	\cdots & {\ul{\pi}_0(X_h)} & {\ul{\pi}_{0}(X)} & {\ul{\pi}_{0}(\wt{X})} & 0 & \cdots \\
	\cdots & {\ul{\pi}_0(X_h)} & {\ul{\pi}_{0}(X^h)} & {\ul{\pi}_{0}(X^t)} & 0 & \cdots
	\arrow[from=1-1, to=1-2]
	\arrow[from=2-1, to=2-2]
	\arrow["\cong", from=1-2, to=2-2]
	\arrow[from=1-2, to=1-3]
	\arrow[from=2-4, to=2-5]
	\arrow[from=1-4, to=1-5]
	\arrow[from=1-3, to=1-4]
	\arrow[from=2-3, to=2-4]
	\arrow[from=1-5, to=2-5]
	\arrow[from=1-4, to=2-4]
	\arrow[from=1-3, to=2-3]
	\arrow[from=1-5, to=1-6]
	\arrow[from=2-5, to=2-6]
	\arrow[from=2-2, to=2-3]
\end{tikzcd}\]
For degree zero the claim follows from this portion of the diagram and the fact that there is a column of zeros.  For degrees above zero the claim follows inductively from the five lemma.
\end{proof}

The following is an corollary of \cref{prop: K theory connected cover,prop: HZ connective cover,lemma: Greenlees-Meier Connected cover lemma}.

\begin{corollary}\label{cor: Tate connected cover}
    The maps $\wt{K_G(k)}\to K_G(k)^t$ and $\wt{\HZ}\to \HZ^t$ are connective coverings.
\end{corollary}

We now prove the second part of \cref{thm: Tate tilde intro version}.

\begin{proposition}\label{prop: tilde equivalence}
    The Postnikov truncation $K_G(k)\to \HZ$ induces an equivalence of $G$-spectra 
    \[
        \wt{K_G(k)}\to \wt{\HZ}.
    \]
\end{proposition}
\begin{proof}
    There is a commutative square of $G$-spectra
    \[
    \begin{tikzcd}
        \wt{K_G(k)} \ar[r] \ar[d] & \wt{\HZ} \ar[d] \\
        K_G(k)^t \ar[r] & \HZ^t
        \end{tikzcd}
    \]
    where both vertical maps are connective covers by \cref{cor: Tate connected cover}.  The result now follows from that fact that the bottom arrow is an equivalence by \cref{prop: Tate equivalence}.
\end{proof}

This completes the proof of \cref{thm: pullback theorem}, which follows immediately from \cref{prop: tilde equivalence,prop: Tate equivalence,lemma: pullback lemma}. 

\subsection{The $RO(G)$-graded $K$-groups of finite fields}

\cref{thm: pullback theorem} tells us there is a homotopy pullback diagram
\[
    \begin{tikzcd}
        K_G(k) \ar[r] \ar[d] & \HZ \ar[d]\\
        K_G(k)^h \ar[r] & \HZ^h,
    \end{tikzcd}
\]
which we use in this section to reduce the $RO(G)$-graded $K$-groups of finite fields to those of $\HZ$. This reduction proceeds by noticing that the pullback square give us an equivalence of fibers
\begin{align*}
    \tau_{\geq 1} K_G(k) \xrightarrow{\sim} \tau_{\geq 1} K_G(k)^h.
\end{align*}
The homotopy Mackey functors of the fiber are straightforward to compute, giving us \cref{thm: fiber intro version} from the introduction.
\begin{proposition}  \label{prop: homotopy of the fiber}
    The $RO(G)$-graded homotopy groups of $\tau_{\geq 1} K_G(k)^h$ are given on orientation preserving representations $V$ by
    \[
    \ul\pi_V \left( \tau_{\geq 1} K_G(k)^h \right)\cong
    \begin{cases}
        \varominus^i & |V|=2i-1>0, \\
        0 & \textrm{otherwise},
    \end{cases}
\]
and on orientation reversing $V$ by
\[
    \ul\pi_V \left( \tau_{\geq 1} K_G(k)^h \right) \cong
    \begin{cases}
        \varoplus^i & |V|=2i-1>0, \\
        0 & \textrm{otherwise}.
    \end{cases}
\]
\end{proposition}
\begin{proof}
    The homotopy fixed points spectral sequences shows that the map $K_G(k)^h\to \HZ^h$ is an equivalence for all $V\in RO(G)$ with $|V|\leq 0$.  Thus for $|V|< 0$ the long exact sequence associated to the fibration 
    \[
        \tau_{\geq 1} K_G(k)^h \to K_G(k)^h\to \HZ^h
    \]
    takes the form
    \[
        \dots \to \ul{\pi}_{V+1}K_G(K)^h\xrightarrow{\cong} \ul{\pi}_{V+1}\HZ^h\to \ul{\pi}_V (\tau_{\geq 1} K_G(k)) \to \ul{\pi}_VK_G(k)^h\xrightarrow{\cong} \ul{\pi}_V\HZ^h\to \dots
    \]
    which shows that $\ul{\pi}_V (\tau_{\geq 1} K_G(k))=0$ when $|V|< 0$.  When $|V|=0$ the claim holds because the Mackey functor $\ul{\pi}_{V+1}\HZ^h=0$.
    
    For $|V|>0$, the long exact sequence takes the form
    \[
        \dots \to 0\to \ul{\pi}_V (\tau_{\geq 1} K_G(k))\to \ul{\pi}_VK_G(k)^h\to 0\to \dots
    \]
    so we have $\ul{\pi}_V (\tau_{\geq 1} K_G(k))\cong \ul{\pi}_VK_G(k)^h$ for $|V|>0$.  Putting all this together with \cref{prop: homotopy fixed point groups} we obtain the result.    
\end{proof}

    Note that the groups appearing in the Mackey functors $\varoplus^i$ and $\varominus^i$ consist of torsion abelian groups with $q$ invertible. If $p$ denotes the characteristic of $k$, it follows that the $p$-completion of all these Mackey functors are zero, hence $(\tau_{\geq 1} K_G(k))^{\wedge}_{p}\simeq 0$. 
    \begin{corollary}
        For $q = p^r$ then there  there is an equivalence $K_{C_n}(\bbF_{q^n})^{\wedge}_{p}\simeq \HZ^{\wedge}_{p}$.
    \end{corollary}

\begin{theorem}\label{thm: K groups split}
    The $RO(G)$-graded homotopy groups of $K_G(k)$ are given by
    \[
        \ul{\pi}_{V}(K_{C_n}(k))\cong \ul{\pi}_{V} \left( \tau_{\geq 1} K_G(k) \right)\varoplus \ul{\pi}_{V}(\HZ).
    \]
    for all $V\in RO(C_n)$.
\end{theorem}
\begin{proof}
    We consider the long exact sequence
    \[
            \dots \to \ul{\pi}_{V+1}\HZ\to \ul{\pi}_{V} (\tau_{\geq 1} K_G(k)) \to \ul{\pi}_VK_G(k)\to \ul{\pi}_V\HZ\to \ul{\pi}_{V-1} (\tau_{\geq 1} K_G(k)) \to \cdots
        \]
        associated to the fibration 
        \[
            \tau_{\geq 1} K_G(k) \to K_G(k)\to \HZ
        \]
        for various choices of $V$.  Note that because $\ul{\pi}_V(\tau_{\geq 1} K_G(k))=0$ for all $V$ with $|V|\leq0$ we immediately obtain the result for $V$ with non-positive total degree.  
        
        When $|V|=2i>0$ we have an exact sequence
        \[
            \ul{\pi}_{V+1}\HZ\to 0 \to \ul{\pi}_V K_{G}(k)\to \ul{\pi}_{V}\HZ\to M^{i}
        \]
        where $M^i$ is either $\varoplus^i$ or $\varominus^i$.  Note that we have \[\ul{\pi}_{V}(\HZ)(G/e) =\pi^e_{V}(\HZ) = \pi_{|V|}H\mathbb{Z}=0.\]  Since both $\varoplus^i$ and $\varominus^i$ are in the image of the right adjoint the the evaluation functor  $\Mack_{G}\to \Mod_{\bbZ G}$, we see that the map $\ul{\pi}_{V}\HZ\to M^{i}$ must be the zero map and we have an isomorphism 
        \[
        \ul{\pi}_V K_{G}(k)\xrightarrow{\cong} \ul{\pi}_{V}\HZ\cong \ul{\pi}_V(\tau_{\geq 1} K_G(k))\oplus \ul{\pi}_{V}\HZ
        \]
        for $|V| =2i>0$.

        For $|V| =2i-1>0$ the long exact sequence has the form
        \[
            \ul{\pi}_{V+1}\HZ\to \ul{\pi}_V(\tau_{\geq 1} K_G(k)) \to \ul{\pi}_V K_{G}(k)\to \ul{\pi}_{V}\HZ\to 0.
        \]
        The same reasoning as the last case tells us that the leftmost map is the zero map so this is really an extension problem of the form
        \[
            0\to\ul{\pi}_V(\tau_{\geq 1} K_G(k)) \to \ul{\pi}_V K_{G}(k)\to \ul{\pi}_{V}\HZ\to 0.
        \]
        and by \cref{prop: no extensions} this extension problem is trivial, giving us the desired splitting.
\end{proof}

\begin{remark}
    While we obtain a splitting of $K$-groups at every degree, we stress that this is not coming from a splitting of $G$-spectra.  Indeed, looking at underlying spectra this would be equivalent to $K(k)$ splitting as $H\mathbb{Z}\vee \tau_{\geq 1} K(k)$, which is not true. 
\end{remark}

\subsection{Geometric fixed points}\label{sec: geometric fixed points}

In this section, we show that calculating the geometric fixed points of $K_G(k)$ can also be reduced to computing the geometric fixed points of $\HZ$. For any finite group $G$ let $\EP$ denote a $G$-space with
\[
(\EP)^H \cong 
\begin{cases}
	* & H\neq G\\
	\emptyset & H=G.
\end{cases}
\] 
This space is unique up to $G$-homotopy equivalence and we write $\tEP$ for the unreduced suspension of $\EP$.  For any $G$-spectrum $X$ we write $X^{\Phi} = \tEP\wedge X$. Note that if $\ell$ is a prime number, and $G$ is a cyclic group of order $\ell$, there is a $G$-homotopy equivalence $\EP\simeq EG$ and thus $X^{\Phi}\cong \wt{X}$, as defined in \cref{sec: Tate}.
\begin{definition}
	The \emph{geometric fixed points} of $X$ are the spectrum $\Phi^G(X) = (X^{\Phi})^G$.   
\end{definition}

\begin{theorem}\label{cor: same geometric fixed points}
    There is an equivalence of spectra $\Phi^G(K_{G}(k))\to \Phi^G(\HZ)$.
\end{theorem}

We thank an anonymous referee for indicating \cref{cor: same geometric fixed points} follows from the equivalence of $G$-spectra $\wt{K_G(k)}\simeq \wt\HZ$ (\cref{prop: tilde equivalence}) and the following lemma.

\begin{lemma}
	For any $G$-spectrum $X$ the canonical map
	\[
		X\to \wt{X}
	\]
	becomes an equivalence of spectra after applying $\Phi^G$.
\end{lemma}
\begin{proof}
	It suffices to check that the map $\tEP\to \tEP\wedge \tEG$, obtained by smashing the canonical map $S^0\to \tEG$ with $\tEP$, is an equivalence of $G$-spectra.  For this, it suffices to check that the fiber $\tEP\wedge EG_+$ is contractible.  For this, consider the fiber sequence
	\[
	\EP_+\wedge EG_+\to EG_+\to \tEP\wedge EG_+.
	\]
	We are done if we show that the left map is an equivalence.  For this, we observe that this map is obtained by applying the suspension functor to the map of $G$-spaces
	\[
	\EP\times EG\to EG
	\]
	obtained by collapsing $\EP$ to a point.  This is an equivariant map between free contractible $G$-spaces and is therefore a $G$-homotopy equivalence.
\end{proof}

Therefore, the computation of $\Phi^G K_G(k)$ reduces to the following well-known computation of $\Phi^G\HZ$. We did not know a complete reference in the literature so we give the computation here. 

\begin{proposition} \label{prop: geometric fixed points of HZ}
    Let $G$ be a finite cyclic group.
    \begin{propenum}
        \item \label{prop: geometric fixed points of HZ for p group}
        If $G$ is a nontrivial cyclic $\ell$-group for a prime $\ell$, then as a graded ring,
        \[
            \pi_* \Phi^G\HZ \cong \bbZ/\ell[x], \qquad |x|=2.
        \]
        \item \label{prop: geometric fixed points of HZ for non p group}
        If $G$ is a cyclic group whose order has at least two prime factors, then 
        \[
            \Phi^G\HZ\simeq 0.
        \]
    \end{propenum}
\end{proposition}

We begin with the case where $|G|$ is a power of a prime $\ell$. The $\ell=2$ case was proven by Hill--Hopkins--Ravenel \cite[Proposition 3.18]{HHR}. We learned of the following proof technique for the general case from a \href{https://mathoverflow.net/a/179933}{MathOverflow answer by Justin Noel}.
 
\begin{proof}[Proof of \cref{prop: geometric fixed points of HZ for p group}]
    First suppose that $G = C_{\ell^n}$ and for all $n$ write $\tEP_{C_n}$ for unreduced suspension of the associated universal space. For all $n>1$ let $f\colon C_{\ell^n}\to C_{\ell}$ denote the quotient by the subgroup $C_{\ell^{n-1}}$.  Checking the fixed points for all subgroups, we see that $f^*\tEP_{C_\ell}$ is a model for $\tEP_{C_{\ell^n}}$.  We have an isomorphism
    \[
        \Phi^{C_{\ell^n}}(\HZ)_{*}\cong \HZ^{C_{\ell^n}}_*(\tEP_{C_{\ell^n}}) = \HZ^{C_{\ell^n}}_*(f^*\tEP_{C_{\ell}})\cong \HZ^{C_\ell}_*(\tEP_{C_\ell})
    \]
    where the last isomorphism comes from an isomorphism of Bredon chains 
    \[
    C_*^{C_{\ell^n}}(f^*\tEP_{C_\ell};\ul{\bbZ}) \cong C_*^{C_\ell}(\tEP_{C_\ell};\ul{\bbZ}).
    \] 
    Thus the computation for $G = C_{\ell^n}$ is reduced to the case $n=1$.
    
    When $n=1$ we have $\wt{E}C_\ell = \tEP_{C_{\ell}}$ and so there is an equality \[\Phi^G(\HZ) = \wt{\HZ}^G. \]
    By \cref{cor: Tate connected cover}, the right hand side is the connective cover of $\HZ^{tG}$.  Using the Tate spectral sequence (\cite[Theorem 10.3]{Greenlees--May})
    \[
        E^2_{r,s} = \widehat{H}^r(C_\ell;\pi^{C_\ell}_s\HZ)\Rightarrow \pi_{r+s}^{C_\ell}\HZ^{tC_\ell}
    \]
    we see that $\pi_*^{C_\ell}\HZ^{tC_\ell}$ is isomorphic, as a graded ring, to the Tate cohomology of $C_\ell$. Taking the non-negative part, we see that $\pi_*\Phi^{C_\ell}(\HZ)$ is isomorphic to the group cohomology ring of $C_\ell$, as claimed.
\end{proof}

Before handing the case where $n$ has two prime factors, we need some notation and a lemma.  The equivariant homotopy groups of $X^{\Phi}$ can be understood as a localization of the homotopy groups of $X$.  Let $\tilde{\rho}$ denote the reduced regular representation of $G$ and let $\alpha\colon S^0\to S^{\wt{\rho}}$ be the inclusion of the the fixed points $\{0,\infty\}$.  Since the reduced regular representation has trivial fixed points the map $\alpha$ is not null homotopic and represents a non-trivial class in $\pi^G_{-\wt{\rho}}(\mathbb{S}_G)$.

\begin{lemma}
	For any $G$-spectrum $X$ the homotopy groups of $\Phi^G(X)$ are given by the localization \[\pi_*(\Phi^G(X))\cong \alpha^{-1}\pi_*^G(X).\]
\end{lemma}
\begin{proof}
	The statement when $G=C_{2^n}$ is  \cite[Proposition 3.18]{HHR}.  The proof in the general case is identical, except that we replace the role of the sign representation $\sigma$ from \cite{HHR} with $\wt{\rho}$.
\end{proof}

\begin{proof}[Proof of \cref{prop: geometric fixed points of HZ for non p group}]
    Now suppose that $G$ is any finite group with at least two distinct prime factors $s$ and $t$. So there exists some subgroups $C_s,C_t\subset G$. Let $\wt{\rho}$ denote the reduced regular representation of $G$ and let $\lambda_s$ and $\lambda_t$ denote the $2$-dimensional representations of $C_s$ and $C_t$, respectively, which rotate the plane by $\frac{2\pi}{s}$ and $\frac{2\pi}{t}$ radians.  Note that $\lambda_s$ and $\lambda_t$ have trivial fixed points and thus the induced representations $\Lambda_s=\mathrm{Ind}_{C_s}^G(\lambda_s)$ and $\Lambda_t=\mathrm{Ind}_{C_t}^G(\lambda_t)$ have trivial $G$-fixed points.  It follows that for a sufficiently large integer $n$ there is an equivariant embedding of $\Lambda_{s}$ and $\Lambda_t$ into $n\wt{\rho}$.  

    For any two $G$-representations $V$ and $W$ we have $\alpha_{V\oplus W}= \alpha_{V}\alpha_W$ so we must have that $\alpha_{\Lambda_s}$ and $\alpha_{\Lambda_t}$ divide $\alpha_{n\wt{\rho}}$.  Thus if we invert $\alpha_{\wt{\rho}}$ we also invert $\alpha_{\Lambda_{s}}$ and $\alpha_{\Lambda_{t}}$.  We claim that, as elements in the ring $\HZ_{\star}$, $\alpha_{\Lambda_{s}}$ is $s^{|G/C_s|}$-torsion and $\alpha_{\Lambda_{q}}$ is $t^{|G/C_t|}$-torsion.  It follows that inverting both kills everything and thus $\Phi^G(\HZ)\simeq *$.

    To prove the claim, note that $\alpha_{\lambda_s}$ and $\alpha_{\lambda_t}$ are $s$ and $t$-torsion, in $\HZ_{\star}$.  This follows, for example, from \cref{theorem: Ferland--Lewis odd} and \cref{theorem: Ferland--Lewis} below which imply that $\HZ^{C_s}_{-\lambda_s}\cong \mathbb{Z}/s$.  Since $\HZ$ is a $G$-$E_{\infty}$ ring spectrum we can take norms and we have $\alpha_{\Lambda_s} = N^{G}_{C_s}(\alpha_{\lambda_s})$. The claim now follows from the multiplicative property of the norms.
\end{proof}

\section{Multiplicative structure}\label{section: multiplicative structure}

In this section we describe multiplicative structure in the $RO(G)$-graded equivariant $K$-theory of finite fields.  We begin by describing the precise algebraic object which encodes these multiplications in \cref{subsec: Einfty Green functor}. In \cref{subsec: products} prove qualitative statements about the multiplication in $\ul\pi_\star K_G(k)$.

\subsection{$\bbE_\infty$-Green functor structure} \label{subsec: Einfty Green functor}

Whereas Mackey functors are the analogues of abelian groups in $G$-equivariant algebra, \emph{Green functors} are the analogues of rings.

\begin{definition}[{\cite[Chapter 2]{Bouc:GreenFunctors}}]\label{definition:Green}
	A \emph{Green functor} for an abelian group $G$ is a $G$-Mackey functor $\ul S$ such that:
	\begin{enumerate}
		\item $\ul S(G/H)$ is a ring for all $H\leq G$, 
		\item the restriction maps are ring homomorphisms,
		\item the Weyl group actions are actions through ring automorphisms, and
		\item (Frobenius reciprocity) whenever it makes sense, we have the relations 
        \[
            T^K_H(x)\cdot y = T^K_H(x\cdot R^K_H(y)), \qquad x\cdot T^K_H(y) = T^K_H(R^K_H(x)\cdot y).
        \]
	\end{enumerate}
\end{definition}

The category $\Mack_G$ of $G$-Mackey functors has a symmetric monoidal product called the \emph{box product}. We denote the box product of two Mackey functors $\ul M$ and $\ul N$ by $\ul M\boxtimes \ul N$. A Green functor is precisely a monoid with respect to the box product \cite[\S 2.3]{Bouc:GreenFunctors}.

Just as the zeroth homotopy group $\pi_0$ of a commutative ring spectrum is a commutative ring, the zeroth homotopy Mackey functor $\ul\pi_0$ of a commutative ring $G$-spectrum is a commutative Green functor. The collection of all homotopy Mackey functors assembles into an \emph{$RO(G)$-graded Green functor}.  Explicitly, this structure consists of a collection of Mackey functors $\ul M_{\star}$, graded on $\star\in RO(G)$, together with maps of Mackey functors
\[
	\ul M_{V}\boxtimes \ul M_W\to \ul M_{V+W}
\]
subject to certain associativity and unitality assumptions. We note also that the properties of the box product imply a graded version of Frobenius reciprocity (part (d) of \cref{definition:Green}) which is completely analogous to the ungraded case.  A full account can be found in a paper of Lewis--Mandell \cite[Section 3]{LewisMandell:universalCoefficients}.   Note that Lewis--Mandell refer to graded Green functors as graded Mackey rings.  

There are several things one might mean by a commutative ring $G$-spectrum.  The possible commutative ring structures correspond to kinds of operads collectively known as \emph{$N_{\infty}$-operads}, as introduced by Blumberg--Hill \cite{BH:Ninfinity}.  On the zeroth homotopy Mackey functor $\ul\pi_0$, an algebra over an $N_{\infty}$-operad admits multiplicative structure called an \emph{incomplete Tambara functor} \cite{BH:incomplete}.
	
In this paper, we only work with algebras over the ``most incomplete'' $N_{\infty}$-operad. Following Barwick--Glasman--Shah \cite{BGS}, we refer to these as \emph{$\bbE_\infty$-Green functors}. The corresponding notion of incomplete Tambara functor has no multiplicative norms, and as such, is equivalent to the notion of a Green functor. Our reason for this choice is that we are not aware of a proof in the literature that equivariant $K$-theory spectra are algebras over any more structured $N_{\infty}$-operads. For completeness, we include a proof, due to Barwick--Glasman--Shah \cite{BGS}, that $K_G(R)$ has the claimed multiplicative structure.

\begin{proposition}[Barwick--Glasman--Shah]
	Let $R$ be a commutative $G$-ring for a finite group $G$. Then $K_G(R)$ admits the structure of a $\bbE_\infty$-Green functor.
\end{proposition}
\begin{proof}
	For the purposes of this proof we use the model for $K_G(R)$ given in \cite[Section 8]{BGS}, where it is called the ``$K$-theory of group actions.''  A proof that this model for equivariant algebraic $K$-theory of a $G$-ring is isomorphic to Merling's in the equivariant stable homotopy category is widely expected, and will appear in forthcoming work of the first-named author with Calle, Chedalavada, and Mejia \cite{Calle--Chan--Chedalavada--Mejia}.  
	
	Given a $G$-ring $R$, let $\Perf_R$ be a the $\infty$-category of perfect modules over the Eilenberg--MacLane spectrum $\mathrm{H}R$ in the $\infty$-category of spectra $\mathrm{Sp}$. This $\infty$-category inherits a $G$-action which, informally, is given as follows: if $M$ is a perfect $R$-module then $^{g}M$ is the $R$-modules whose action is given by the composite
	\[
		R\wedge M\xrightarrow{g\wedge 1} R\wedge M\to M,
	\]
	where the second map is the action of $R$ on $M$.
	
	If $R$ is an $\mathbb{E}_{\infty}$-ring, then $\Perf_R$ is an $\mathbb{E}_{\infty}$-algebra in the $\infty$-category $\mathrm{Fun}(BG,\mathrm{Cat}^{\mathrm{perf}}_{\infty})$ of perfect stable $\infty$-categories with a $G$-action. Barwick--Glasman--Shah show \cite[Proposition 8.2]{BGS}, there is a lax monoidal functor of $\infty$-categories
	\[
		K_G\colon \mathrm{Fun}(BG,\mathrm{Cat}^{\mathrm{perf}}_{\infty})\to \Mack(\mathrm{Sp})
	\]
	where the target is Barwick's $\infty$-category of spectral Mackey functors \cite{Barwick}.  The $K$-theory of group actions is exactly $K_G(\mathrm{Perf}_R)$.  Since lax monoidal functors preserve $\mathbb{E}_{\infty}$-algebras, the claim follows.
\end{proof}

\subsection{Products in equivariant $K$-theory} \label{subsec: products}

In this section, we describe how to compute multiplications in the $RO(G)$-graded equivariant $K$-theory of finite fields. 

To get a sense of the expected multiplicative behavior in $K_G(\bbF_{q^n})$, we first discuss the ring structure of the nonequivariant $K$-theory of finite fields. There, all products are essentially determined for degree reasons. Indeed, the $K$-theory in nonzero even degrees vanishes, so the product of any two classes in nonzero degrees vanishes. Therefore, the truncation map 
\[
    K(\bbF_q) \to H\bbZ
\]
exhibits the graded ring $K_*(\bbF_q)$ as a square-zero extension of $\pi_* H\bbZ$. This section will culminate in \cref{theorem: square zero extension}, which shows that this is true equivariantly: $\pi_\star^H K_G(k)$ is a square zero extension of $\HZ_\star^H$ for all $H\leq G$.

We keep the notation of \cref{sec: proof}, whereby $k$ is a degree $n$ Galois extension of $\bbF_q$ with Galois group $G\cong C_n$. We write $\tau_{\geq 1} K_G(k)$ for the fiber of the truncation map $\tau_{\leq 0}\colon K_G(k)\to \HZ$. For a subgroup $H\leq G$, the direct sum decomposition of $RO(G)$-graded rings
\[
	\pi_{\star}^H K_{C_n}(k)\cong \pi_{\star}^H \left( \tau_{\geq 1} K_G(k) \right)\oplus \HZ_{\star}^H
\]
of \cref{thm: K groups split} reduces the possible ways that classes can multiply together to the cases where each multiplicand is contained in exactly one of the direct summands. 

\begin{proposition} \label{prop: products}
	For a subgroup $H\leq G$, let $x\in \pi_{V}^HK_G(k)$ and $y\in \pi_{W}^HK_G(k)$ for virtual representations $V,W\in RO(G)$. The multiplication in $\pi^H_{\star}(K_G(k))$ satisfies the following.
    \begin{propenum}
        \item \label{prop: square zero part}
        If $x$ and $y$ are both elements in $\pi_{\star}^H \left( \tau_{\geq 1} K_G(k) \right)$ then $xy=0$.

        \item \label{prop: mixed products away from zero}
        If $x\in \pi_{V}^H \left( \tau_{\geq 1} K_G(k) \right)$ and $y\in \HZ^H_W$ with $|W|\neq 0$ then $xy=0$.

        \item \label{prop: products in HZ}
        If $x,y\in \HZ_{\star}$ then $xy\in \HZ_{\star}$ and the product is identified with the product in the ring structure of $\HZ_{\star}$.
    \end{propenum}
\end{proposition}

\begin{proof}[Proof of \cref{prop: square zero part}]
	Note that the homotopy Mackey functors in $\ul{\pi}_{\star}(\tau_{\geq 1} K_G(k))$, given by either $\varoplus^i$ or $\varominus^i$ satisfy the property that the transfer maps are surjective.  With this, we may write $x = T_e^H(x')$ for some element $x'\in \pi^e_{V}(K_{G}(k))\cong K_{|V|}(k)$ where $|V|$ is the total virtual dimension.  By Frobenius reciprocity we have
	\[
		xy = T^H_e(x')y = T^H_e(x'R^H_e(y))
	\]
	and we can compute the product $x'R^H_e(y)$ using the multiplicative structure for the underlying ring spectrum $K(k)$. All such products are zero for degree reasons.
\end{proof}

The same argument, together with the fact that $\HZ^e_{W}=0$ for $|W|\neq 0$ proves \cref{prop: mixed products away from zero}.

The case where $|W|\neq 0$ is handled similarly. Indeed, if $x\in \pi_V^H(\tau_{\geq 1} K_G(k))$ and $y\in \HZ^H_{W}$ with $|W|=0$ then again we have $xy = T^H_e(x'R^H_e(y))$, in which case we have reduced the computation to the case where $H=$. In this case, $\HZ^e_W\cong \mathbb{Z}$, $y$ is an integer, and $xy$ is given by the canonical action of the integers on $\pi_V^H(\tau_{\geq 1} K_G(k))$.

All that remains is the case where both $x$ and $y$ come from the $\HZ_{\star}$ components.

\begin{proof}[Proof of \cref{prop: products in HZ}]
	The second claim follows from the first because the projection map $\pi_{\star}^H (K_G(k))\to \HZ_{\star}^H$ is induced by the truncation map $K_G(k)\to \HZ$, which is a map of ring spectra \cite[Prop. 4.35]{HHR}.  Thus it suffices to check that the component of $xy$ in $\pi_{V+ W}^H(\tau_{\geq 1}K_G(k))$ is zero.  We will write $(xy)_{\geq 1}$ for this element.
	
	If $|V|=|W|=0$ then  $\pi_{V+W}(\tau_{\geq 1}K_G(k))=0$ by \cref{prop: homotopy of the fiber} so $(xy)_{\geq 1}=0$.  Thus it suffices to consider one of $|V|$ or $|W|$ is not zero.  Without loss of generality say $|V|\neq 0$.  In this case, the element $(xy)_{\geq 1}$ is in the image of the composite
	\begin{equation} \label{equation: multiplication for HZ HZ case}
		\HZ_{V}\boxtimes\HZ_{W} \to \ul\pi_{V+W}(K_G(k)) \to \ul\pi_{V+W}(\tau_{\geq 1}K_G(k))
	\end{equation}
	of the multiplication map followed by the projection. Note that $(\HZ_V\boxtimes \HZ_W)(G/e)\cong \HZ^e_{V}\otimes \HZ^e_W$, which is the zero group because $|V|\neq 0$.  But $\pi_{V+W}(\tau_{\geq 1}K_G(k))$ is either $\varoplus^i$, $\varominus^i$, or zero so in all cases there is a $G$-module $M$ so that $\pi_{V+W}(\tau_{\geq 1}K_G(k))\cong \ul R(M)$, where $\ul R$ is the right adjoint to the functor $\Mack_G\to \Mod_{\bbZ G}$ which evaluates at $G/e$.  In particular, since $(\HZ_{V}\boxtimes\HZ_{W})(G/e)=0$ the composite \eqref{equation: multiplication for HZ HZ case} is the zero map of Mackey functors so $(xy)_{\geq 1}=0$.
\end{proof}

By \cref{prop: products in HZ} we have that $\HZ_{\star}^{H}$ is a subalgebra of $\pi_{\star}^H(K_G(k))$.  We see that $\pi_{\star}^H(\tau_{\geq 1}K_G(k))$ is a submodule, and by \cref{prop: square zero part} we identify $\pi_{\star}^H K_G(k)$ is a square zero extension.

\begin{theorem}\label{theorem: square zero extension}
	For any $H\leq G$, $\pi_{\star}^H(K_G(k))$ is a square zero-extension of $\HZ_{\star}^{H}$.
\end{theorem}

The ring $\HZ^{G}_{\star}$ is quite complicated. Already for the group $G =C_2$, $\HZ_\star^{C_2}$ is infinitely generated and non-noetherian. A presentation of this ring can be found in work of Greenlees \cite[Corollary 2.6]{Greenlees:four-approaches} and Zeng \cite[Proposition 6.5]{Zeng}. In the next section we use these presentations to give a presentation of the $RO(C_2)$-graded ring $\pi_{\star}^{C_2}(K_{C_2}(\mathbb{F}_{q^2}))$.

\section{Extensions of prime degree} \label{sec: examples}

In this section we specialize to the case of $k=\bbF_{q^\ell}$ as a degree $\ell$ extension of $\bbF_q$ for $\ell$ a prime.  We recall the computation of $\HZ_{\star}$ and use this to give an explicit identification of the $RO(C_{\ell})$-graded $K$-groups.  The computation is slightly different when $\ell=2$ so we treat this case separately.

\subsection{Quadratic extensions}
    Our first example is the case of quadratic extensions $\bbF_{q^2}/\bbF_q$ where the Galois group is $C_2$.  \cref{table C2 Mackeys}  below gives Lewis diagrams for all the relevant $C_2$-Mackey functors.

\begin{table}[hbp]
    \setlength{\tabcolsep}{5mm} 
    \def\arraystretch{1.25} 
    \centering
    \begin{tabularx}{0.7\textwidth}{XXXXXXXXXX}
        & & & & & & & & \\[-18pt] \hline 
        \multicolumn{2}{|c|}{$\boxempty$} & 
        \multicolumn{2}{|c|}{$\olboxempty$} & 
        \multicolumn{2}{|c|}{$\circ$} & 
        \multicolumn{2}{|c|}{$\boxslash$} & 
        \multicolumn{2}{|c|}{$\olboxslash$}
        \\ \hline 

        \multicolumn{2}{|c|}{
            \begin{tikzcd} \bbZ \ar[d,shift right,"1"'] \\ \bbZ \ar[u,shift right, "2"']  \end{tikzcd}
        } & 
        \multicolumn{2}{|c|}{
            \begin{tikzcd} 0 \ar[d,shift right,"0"'] \\ \bbZ^{\sigma} \ar[u,shift right, "0"']  \end{tikzcd}
        } & 
        \multicolumn{2}{|c|}{
            \begin{tikzcd} \bbZ/2 \ar[d,shift right,"0"'] \\ 0 \ar[u,shift right, "0"']  \end{tikzcd}
        } &
        \multicolumn{2}{|c|}{
            \begin{tikzcd} \bbZ \ar[d,shift right,"2"'] \\ \bbZ \ar[u,shift right, "1"']  \end{tikzcd}
        } &
        \multicolumn{2}{|c|}{
            \begin{tikzcd} \bbZ/2 \ar[d,shift right,"0"'] \\ \bbZ^{\sigma} \ar[u,shift right, "1"'] \end{tikzcd}
        } \\ \hline\hline

        \multicolumn{5}{|c|}{$\varominus^i$} & \multicolumn{5}{|c|}{$\varoplus^i$} 
        \\ \hline

        \multicolumn{5}{|c|}{
            \begin{tikzcd} \bbZ/(q^i-1) \ar[d,shift right,"1+q^i"'] \\ \bbZ/(q^{2i}-1) \ar[u,shift right, "1"']  \end{tikzcd}
        } & 
        \multicolumn{5}{|c|}{
            \begin{tikzcd} \bbZ/(-q^i-1) \ar[d,shift right,"1-q^i"'] \\ \bbZ/(q^{2i}-1) \ar[u,shift right, "1"']  \end{tikzcd}
        } \\ \hline
    \end{tabularx}
    \caption{$C_2$-Mackey functors which appear in the $RO(C_2)$-graded homotopy of $K_{C_2}(\mathbb{F}_{q^2})$.}
    \label{table C2 Mackeys}
\end{table}

\subsubsection{The $RO(C_2)$-graded equivariant $K$-groups $\ul\pi_\star K_{C_2}(k)$}
    
We begin by recalling the $RO(C_2)$-graded homotopy groups of $\HZ$.  This computation is originally due to unpublished work of Stong; we use \cite{Ferland--Lewis} as a reference. We recall this computation following the \textit{motivic} grading convention. There are two irreducible real orthogonal $C_2$-representations: the 1-dimensional trivial representation 1 and the 1-dimensional sign representation $\sigma$. As a result, $\ul\pi_\star\HZ$ is a bigraded Mackey functor. We write $(x,y)$ for the bidegree corresponding to the representation $(x-y)+y\sigma$, whereby $x$ is the \textit{total degree} and $y$ is the \textit{twisted degree}.

\begin{theorem}[{\cite[Theorem 8.1]{Ferland--Lewis}}]\label{theorem: Ferland--Lewis}
	Let $(x,y)\in RO(C_2)$.  The $RO(C_2)$-graded homotopy groups of $\HZ$ are given by the following rules:
	\begin{enumerate}
		\item If $x=y$ then
		\[
		\ul{\pi}_{x,y}\HZ\cong \begin{cases}
			0 & x>0,\\
			\boxempty & x=0,\\
			\circ & x<0.
		\end{cases}
		\]
		\item If $y<0$ and $x=0$ then
		\[
		\ul{\pi}_{x,y}\HZ\cong \begin{cases}
			\boxempty & y\ \mathrm{ even},\\
			\olboxempty & y\ \mathrm{ odd},\\
		\end{cases}
		\]
		\item If $y>0$ and $x=0$ then
		\[
		\ul{\pi}_{x,y}\HZ\cong \begin{cases}
			\boxslash & y\ \mathrm{ even},\\
			\olboxslash & y\geq 3\  \mathrm{ odd},\\
			\olboxempty & y=1
		\end{cases}
		\]
		\item If $x>y$ and $x<0$ then
		\[
		\ul{\pi}_{x,y}\HZ\cong \begin{cases}
			\circ & x-y\ \mathrm{ even},\\
			0 & x-y\ \mathrm{ odd},\\
		\end{cases}
		\]
		\item If $y\geq x+3$ and $x>0$ then
		\[
		\ul{\pi}_{x,y}\HZ\cong \begin{cases}
			\circ & x-y\ \mathrm{ odd},\\
			0 & else
		\end{cases}
		\]
		\item If $x-y$ and $x$ are both positive or both negative then $\ul{\pi}_{x,y}\HZ=0$.
	\end{enumerate}
\end{theorem}

We depict this result graphically in \cref{figure: C2 HZ}, along with multiplicative structure we will discuss in the next section.

\begin{figure}[hbtp] 
    \centering
    
    \begin{sseqpage}[grid = none, 
        classes = {draw = none}, 
        title = $\ul\pi_{x,y}\HZ$,
        xrange = {-6}{6}, x tick step = 2,
        yrange = {-6}{6}, y tick step = 2]
        
        \class["\boxempty", name = 1, show name = {right = 0pt}](0,0)
        \class["\circ", name = \alpha, show name = {above = 2pt}](-1,-1)
        \structline
        \foreach \i in {2,3,4,5,6,7} {
            \class["\circ", name = {\alpha^\i}, show name = {above = 2pt}](-\i,-\i)
            \structline
        }

        \class["\boxempty", name = {u}, show name = {right = 0pt}](0,-2)
        \DoUntilOutOfBounds{
            \class["\circ"](\lastx-1,\lasty-1)
            \structline
        }

        \class["\boxempty", name = {u^2}, show name = {right = 0pt}](0,-4)
        \DoUntilOutOfBounds{
            \class["\circ"](\lastx-1,\lasty-1)
            \structline
        }

        \class["\boxempty", name = {u^3}, show name = {right = 0pt}](0,-6)
        \DoUntilOutOfBounds{
            \class["\circ"](\lastx-1,\lasty-1)
            \structline
        }

        \class["\olboxempty"](0,1)
        \DoUntilOutOfBounds{
            \class["\olboxempty"](\lastx,\lasty-2)
        }

        \class["\boxslash", name = {\frac{2}{u}}, show name = {left = 0pt}](0,2)
        \class["\boxslash", name = {\frac{2}{u^2}}, show name = {left = 0pt}](0,4)
        \class["\boxslash", name = {\frac{2}{u^3}}, show name = {left = 0pt}](0,6)

        \class["\olboxslash", name = {y_{1,1}}, show name = {left = 0pt}, red](0,3)
        \foreach \i in {2,3,4,5} {   
            \class["\circ", name = {y_{\i,1}}, show name = {below right = 0pt}, red](\i-1,\i+2)
            \structline
        }

        \class["\olboxslash", name = {y_{1,2}}, show name = {left = 0pt}, red](0,5)
        \DoUntilOutOfBounds{
            \class["\circ", red](\lastx+1,\lasty+1)
            \structline
        }
        
    \end{sseqpage}
    \caption{The $RO(C_2)$-graded homotopy Mackey functors of $\HZ$. Names for some generators at the $C_2/C_2$ level are shown. Lines indicate multiplication by $\alpha$. The square-zero summand (see \cref{proposition: HZ star is a square zero extension}) $M$ is shown in red.}
    \label{figure: C2 HZ}
\end{figure}

With the homotopy groups of $\HZ$ and $\tau_{\geq 1} K_{C_2}(k)$ in hand we can read off the $RO(C_2)$-graded homotopy groups of $K_{C_2}(k)$ using \cref{thm: K groups split}.

\begin{theorem}\label{thm: p equals 2 computation}
    The $RO(C_2)$-graded homotopy Mackey functors of $K_{C_2}(\bbF_{q^2})$ admit a direct sum decomposition
    \[
        \ul{\pi}_{x,y}K_{C_2}(\bbF_{q^2})\cong \ul{\pi}_{x,y} \left( \tau_{\geq 1} K_G(k) \right) \oplus \ul{\pi}_{x,y}\HZ.
    \]
    Explicitly, they are given by
    \[
        \ul{\pi}_{x,y}K_{C_2}(\bbF_{q^2}) = \begin{cases}
            0 & x=2i>0\ and\ x-y>-3\ or\ even\\
            \circ & x=2i>0\ and\ x-y<-3\ and\ odd\\
            \circ\oplus \varominus^i & x=2i-1>0\ and\ x-y\leq-3\ and\ odd\\
            \varominus^i & x=2i-1>0\ and\ x-y>-3\ and\ odd\\
            \varoplus^i & x=2i-1>0\ and\ x-y\ even\\
            \ul\pi_{x,y}\HZ & else
        \end{cases}
    \]
\end{theorem}

We depict this result graphically in \cref{figure: C2 K}. 

\SseqNewClassPattern{circominus}{
    (0,0);
    (-0.275,-0.075)(0.125,0);
}

\begin{figure}[hbtp] 
    \centering
    
    \begin{sseqpage}[grid = none, 
        classes = {draw = none}, 
        class pattern = circominus,
        title = $\ul\pi_{x,y}K_{C_2}(\bbF_{q^2})$,
        xrange = {-6}{6}, x tick step = 2,
        yrange = {-6}{6}, y tick step = 2]
        
        \class["\boxempty"](0,0)
        \class["\circ"](-1,-1)
        \foreach \i in {2,3,4,5,6,7} {
            \class["\circ"](-\i,-\i)
        }

        \class["\boxempty"](0,-2)
        \DoUntilOutOfBounds{
            \class["\circ"](\lastx-1,\lasty-1)
        }

        \class["\boxempty"](0,-4)
        \DoUntilOutOfBounds{
            \class["\circ"](\lastx-1,\lasty-1)
        }

        \class["\boxempty"](0,-6)
        \DoUntilOutOfBounds{
            \class["\circ"](\lastx-1,\lasty-1)
        }

        \class["\olboxempty"](0,1)
        \DoUntilOutOfBounds{
            \class["\olboxempty"](\lastx,\lasty-2)
        }

        \class["\boxslash"](0,2)
        \class["\boxslash"](0,4)
        \class["\boxslash"](0,6)

        \class["\olboxslash", red](0,3)
        \foreach \i in {2,3,4,5} {   
            \class["\circ", red](\i-1,\i+2)
        }

        \class["\olboxslash", red](0,5)
        \DoUntilOutOfBounds{
            \class["\circ", red](\lastx+1,\lasty+1)
        }

        \foreach \i in {1,2,3,4} {
            \class["\varominus^{\i}", blue](2*\i-1,0)
            \DoUntilOutOfBounds{
                \class["\varominus^{\i}",blue](\lastx,\lasty+2)
            }
            \DoUntilOutOfBounds{
                \class["\varominus^{\i}",blue](\lastx,\lasty-2)
            }

            \class["\varoplus^{\i}", blue](2*\i-1,1)
            \DoUntilOutOfBounds{
                \class["\varoplus^{\i}",blue](\lastx,\lasty+2)
            }
            \DoUntilOutOfBounds{
                \class["\varoplus^{\i}",blue](\lastx,\lasty-2)
            }
        }
    \end{sseqpage}
    \caption{The $RO(C_2)$-graded homotopy Mackey functors of $K_{C_2}(\bbF_{q^2})$. The contribution from the square-zero summand $M$ of $\HZ_\star$ (see \cref{proposition: HZ star is a square zero extension}) is red, and the summand $\ul\pi_\star \left( \tau_{\geq 1} K_G(k) \right)$ is blue. Bidegrees (e.g., (1,4)) in which two Mackey functors appear side-by-side represent the sum of the two Mackey functors.}
    \label{figure: C2 K}
\end{figure}

\subsubsection{$RO(C_2)$-graded ring $\pi_\star^{C_2} K_{C_2}(k)$}

Finally, we can describe the graded ring $\pi_{\star}^{C_2} K_{C_2}(k)$.  First, we recall the ring structure of $\HZ^{C_2}_{\star}$ which is described in \cite[Section 2]{Greenlees:four-approaches} and \cite[Proposition 6.5]{Zeng}.  We begin with an $RO(C_2)$-graded ring 
\begin{gather*}
	B = \bbZ\left[u,\alpha,\frac{2}{u^m}\right]/(2\alpha)\\
	|u| = (0,-2) \quad |\alpha| = (-1,-1)
\end{gather*}
where $m$ runs over all positive integers. Note that we are interpretting $\frac{2}{u^m}$ as purely formal; the element $\frac{1}{u^m}$ does not exist.  We are using the motivic grading so that $(a,b)$ corresponds to the virtual representation $(a-b)+b\sigma$. Let $M$ be the $B$-module
\[
	M = \bbF_2[y_{j,k}]_{j,k>0}=\bbF_2[\Sigma^{-1}\alpha^{-j}u^{-k}]_{j,k>0} 
\]
where $|y_{j,k}|=|\Sigma^{-1}\alpha^{-j}u^{-k}| = (j-1,j+2k)$ and the multiplication the element $u$ and $\alpha$ in $B$ works as indicated by the exponents, with the understanding that if either $j$ or $k$ becomes non-negative then the element is zero.  Because $M$ is 2-torsion, multiplication by all elements of the form $\frac{2}{u^i}$ are zero. 

\begin{proposition}\label{proposition: HZ star is a square zero extension}
	The ring $\HZ^{C_2}$ is isomorphic to the square zero extension $B\oplus M$.
\end{proposition}

We can now describe the $RO(C_2)$-graded ring $\pi_{\star}^{C_2}(K_{C_2}(\bbF_{q^2}))$. We write 
\[ 
    N = \pi^{C_2}_{\star}(\tau_{\geq 1}K_{C_2}(k))
\]
for the elements which do not come from $\HZ^{C_2}_{\star}$.  The elements of $N$ form a module over $\HZ_{\star}^{C_2}\subset \pi_{\star}^{C_2}(K_{C_2}(\bbF_{q^2}))$, and by \cref{prop: mixed products away from zero,prop: square zero part}, the ring $\pi_{\star}^{C_2}(K_{C_2}(k))$ is in fact a square zero extension $B\oplus (M\oplus N)$ of $B$. Thus we are done as soon as we describe the action of $B$ on $N$.  

By \cref{prop: homotopy of the fiber} we have an $RO(C_2)$-graded decomposition
\[
	N\cong \bigoplus\limits_{(a,b)\in \bbZ^2}N_{a,b}
\]
where 
\[
	N_{a,b}\cong \begin{cases}
		\bbZ/(q^i-1) & \textrm{$a=2i-1$ for $i>0$ and $b$ even}\\
		\bbZ/(q^i+1) & \textrm{$a=2i-1$ for $i>0$ and $b$ odd}\\
		0 & \textrm{else}.
	\end{cases}
\]
For any $i>0$ let us write $x_{i,b}$ for chosen generators of $N_{2i-1,b}$.  Note that we are free to choose this element to the be transfer of the generator $1\in \pi_{2i-1,b}^e K_{C_2}(\bbF_{q^2})\cong \mathbb{Z}/(q^{2i}-1)$.  The next lemma completely describes the module structure of $N$ over $B$.

\begin{lemma}
	For any $i>0$ and $b\in \bbZ$ we have 
	\[
		ux_{i,b} = x_{i,b-2},\quad \frac{2}{u^m}x_{i,b} = 2x_{i,b+2m},\quad \textrm{and} \quad \alpha x_{i,b}=0.
	\]
\end{lemma}
\begin{proof}
	The identity $\alpha x_{i,b}=0$ follows from \cref{prop: mixed products away from zero} because the total dimension of $\alpha$ is not zero.  
	
	To show that 	$ux_{i,b} = x_{i,b-2}$, we note that $u$ is the generator of the $C_2/C_2$ level of $\HZ^{C_2}_{(0,-2)}\cong \boxempty$.  In particular, the restriction of $u$ is the element $1\in \boxempty(C_2/e)=\bbZ$.  Now we have
	\[
		ux_{i,b} = uT^{C_2}_e(1) = T^{C_2}_e(R^{C_2}_e(u)1) = T^{C_e}_e(1) = x_{i,b-2}.
	\]
	Multiplying by $2u^{m-1}$ gives $2u^mx_{i,b} = 2x_{i,b-2m}$, which we rearrange to $\frac{2}{u^m}x_{i,b-2} = 2x_{i,b}$.
\end{proof}

From this lemma, we see that the generators depend on the characteristic of $\bbF_{q^2}$. Note that $x_{i,b}$ and $x_{i,b+1}$ will generate $x_{i,b'}$ for any $b'<b$. If $q=2^r$ for some $r$ then $2$ is invertible in $\bbZ/(q^i-1)$ and $\bbZ/(q^i+1)$ and $2x_{i,b}$ is a generator of $N_{2i-1,b}$. On the other hand, if $q = p^r$ for $p>2$ then $2$ is not invertible in $\bbZ/(q^i-1)$ nor $\bbZ/(q^i+1)$ and we have $2x_{i,b}$ is not a generator of $N_{2i-1,b}$.  

\begin{corollary}
	If $q=2^r$ then $N$ is generated by $\{x_{i,0},x_{i,1}\}$ for $i>0$. If $q=p^r$ for $p>2$ then $N$ is generated by $\{x_{i,b}\}$ for $i>0$ and $b\in \mathbb{Z}$.
\end{corollary}

We gather together the above observations in the following theorem.

\begin{theorem} \label{theorem: ring structure for C2}
	The ring $\pi_{\star}^{C_2}(K_{C_2}(\mathbb{F}_{q^2}))$ has the presentation
	\[
		\mathbb{Z}\left[u,\alpha,\frac{2}{u^m},x_{i,b},y_{j,k} \right]/I
	\]
	where $I$ is the ideal generated by the relations
	\begin{gather*}
		2\alpha=0,\quad 
		\alpha y_{j,k} = y_{j-1,k},\quad 
		uy_{j,k} = y_{j,k-1},\quad
		 y_{j,k}y_{j',k'}=0, \quad
		 2y_{j,k}=0,\\ 
		ux_{i,b} = x_{i,b-2}, \quad
		 x_{i,b}y_{j,k}=0,\quad 
		 x_{i,b}x_{i',b'}=0,\quad 
		 \alpha x_{i,b}=0,\quad
		 (q^{i}+(-1)^b)x_{i,b}=0
	\end{gather*}
	The indices $i$, $j$, $k$, and $m$ are positive integers.  When $q=p^r$ for $p\neq 2$ then $b$ runs over all integers.  When $q=2^r$ then $b\in\{0,1\}$.

	The (motivic) bi-gradings are 
	\[
		|\alpha| = (-1,-1),\quad |u| = (0,-2),\quad |x_{i,b}| = (2i-1,b),\quad |y_{j,k}| = (j-1,j+2k).
	\]
\end{theorem}

\begin{remark}
	The attentive reader will note that the Koszul rule does not appear explicitly in this presentation.  This is a coincidence, coming from the fact that the product of any two elements which both have odd total degree is zero.
\end{remark}

\subsection{Extensions of odd prime degree}

\subsubsection{The $RO(C_\ell)$-graded equivariant $K$-groups $\ul\pi_\star K_{C_\ell}(k)$}

Here we compute the $RO(C_{\ell})$-graded homotopy of $K_{C_{\ell}}(\bbF_{q^{\ell}})$ for an odd prime $\ell$.  The $RO(C_{\ell})$-graded homotopy of $\HZ$ is similar to the case $\ell=2$, except it is a bit simpler because all non-trivial irreducible representations have dimension $2$. Lewis diagrams for the relevant Mackey functors can be found in \cref{table: Cp Mackey}

\begin{table}[h]
\setlength{\tabcolsep}{5mm} 
\def\arraystretch{1.25} 
\centering
\begin{tabular}{|c|c|c|c|c|}
  \hline
    $\boxempty$  & $\circ$   & $\boxslash$ & $\varominus^i$ 
    \\ \hline 
    \begin{tikzcd} \bbZ \ar[d,shift right,"1"'] \\ \bbZ \ar[u,shift right, "\ell"']  \end{tikzcd}  
    &  
    \begin{tikzcd} \bbZ/\ell \ar[d,shift right,"0"'] \\ 0 \ar[u,shift right, "0"']  \end{tikzcd}   
    & 
    \begin{tikzcd} \bbZ \ar[d,shift right,"\ell"'] \\ \bbZ \ar[u,shift right, "1"']  \end{tikzcd}
     &  
    \begin{tikzcd} \bbZ/(q^i-1) \ar[d,shift right,"\sum\limits_{j=0}^{\ell-1}q^{ji}"'] \\ \bbZ/(q^{\ell i}-1) \ar[u,shift right, "1"']  \end{tikzcd} 
     \\ \hline
  \end{tabular}\\
  \caption{$C_{\ell}$-Mackey functors which appear in the $RO(C_{\ell})$-graded homotopy of $K_{C_{\ell}}(\mathbb{F}_{q^{\ell}})$ for $\ell$ an odd prime.}
  \label{table: Cp Mackey}
\end{table}
\begin{theorem}[{\cite[Theorem 8.1]{Ferland--Lewis}}]\label{theorem: Ferland--Lewis odd}
        Let $\ell$ be an odd prime and let $V\in RO(C_{\ell})$.  The $RO(C_{\ell})$-graded homotopy groups of $\HZ$ are given by the following rules:
        \begin{enumerate}
        \item If $|V^{C_{\ell}}|=0$ then
        \[
            \ul{\pi}_V\HZ\cong \begin{cases}
                 0 & |V|>0,\\
                 \boxempty & |V|=0,\\
                 \circ & |V|<0.
            \end{cases}
        \]
        \item If $|V^{C_{\ell}}|>0$ and $|V|=0$ then $\ul{\pi}_V\HZ\cong \boxempty$.
       
        \item If $|V^{C_{\ell}}|<0$ and $|V|=0$ then $\ul{\pi}_V\HZ\cong \boxslash$.
        \item If $|V^{C_{\ell}}|>0$ and $|V|<0$ then
        \[
            \ul{\pi}_V\HZ\cong \begin{cases}
                 \circ & |V^{C_{\ell}}|\ \mathrm{ even},\\
                 0 & |V^{C_{\ell}}|\ \mathrm{ odd},\\
            \end{cases}
        \]
        \item If $|V^{C_{\ell}}|<0$ and $|V|>0$ then
        \[
            \ul{\pi}_V\HZ\cong \begin{cases}
                 \circ & |V^{C_{\ell}}|\leq-3\ \mathrm{ odd},\\
                 0 & else
            \end{cases}
        \]
        \item If $|V^{C_{\ell}}|$ and $|V|$ are both positive or both negative then $\ul{\pi}_V\HZ=0$.
        \end{enumerate}
    \end{theorem}

    Once again we can immediately obtain the $RO(C_{\ell})$-graded $K$-groups using \cref{thm: K groups split}.
    \begin{theorem}\label{thm: odd prime computation}
        The $RO(C_{\ell})$-graded homotopy Mackey functors of $K_{C_{\ell}}(\bbF_{q^{\ell}})$ admit a direct sum decomposition
        \[
            \ul{\pi}_VK_{C_{\ell}}(\bbF_{q^{\ell}})\cong \ul{\pi}_V \left( \tau_{\geq 1} K_G(k) \right)\oplus \ul{\pi}_V\HZ.
        \]
        Explicitly, they are given by
        \[
            \ul{\pi}_VK_{C_{\ell}}(\bbF_{q^{\ell}}) = \begin{cases}
                0 & |V|=2i>0\ and\ |V^{C_{\ell}}|>-3\ \mathrm{or\ even}\\
                \circ & |V|=2i>0\ and\ |V^{C_{\ell}}|<-3\ \mathrm{and\ odd}\\
                \circ\oplus \varominus^i & |V|=2i-1>0\ and\ |V^{C_{\ell}}|\leq-3\ \mathrm{and\ odd}\\
                \varominus^i & |V|=2i-1>0\ and\ |V^{C_{\ell}}|>-3\ \mathrm{and\ odd}\\
                \ul\pi_V\HZ & \mathrm{else}.
            \end{cases}
        \]
    \end{theorem}

\small
\bibliographystyle{alpha}
\bibliography{main}

\noindent \textsc{Department of Mathematics, Michigan State University, East Lansing, MI 48824} \\
\noindent \emph{Email address:} \texttt{chandav2@msu.edu}

\vspace{\medskipamount}

\noindent \textsc{Department of Mathematics, Cornell University, Ithaca, NY 14853} \\
\noindent \emph{Email address:} \texttt{cpv29@cornell.edu}

\end{document}